\documentclass{siamart1116}
\usepackage{amsmath, amssymb, bm, mathrsfs}
\usepackage{cite, color, graphicx,subcaption,dsfont}
 \makeatletter
\renewcommand{\theequation}{%
	\thesection.\arabic{equation}}
\@addtoreset{equation}{section}
\makeatother

\usepackage{lipsum}
\usepackage{amsfonts}
\usepackage{graphicx}
\usepackage{epstopdf}
\usepackage{algorithmic}
\ifpdf
  \DeclareGraphicsExtensions{.eps,.pdf,.png,.jpg}
\else
  \DeclareGraphicsExtensions{.eps}
\fi

\numberwithin{theorem}{section}

\newcommand{\TheTitle}{Event-triggered  Control 
of Infinite-dimensional Systems} 
\newcommand{\TheAuthors}{Masashi Wakaiki and Hideki Sano}

\headers{Infinite-dimensional Event-triggered Systems}{\TheAuthors}

\title{{\TheTitle}\thanks{Submitted to the editors DATE.
\funding{This work was supported by JSPS KAKENHI Grant Numbers JP17K14699.}}}

\author{
  Masashi Wakaiki and Hideki Sano\thanks{Graduate School of System Informatics, Kobe University, Nada, Kobe, Hyogo 657-8501, Japan
    (\email{wakaiki@ruby.kobe-u.ac.jp}, \email{sano@crystal.kobe-u.ac.jp}).}
}

\usepackage{amsopn}

\ifpdf
\hypersetup{
  pdftitle={\TheTitle},
  pdfauthor={\TheAuthors}
}
\fi

\newtheorem{assumption}[theorem]{Assumption}
\newtheorem{example}[theorem]{Example}
\newtheorem{remark}[theorem]{Remark}

\newcommand{\re}{\mathop{\rm Re}\nolimits}

\begin{document}

\maketitle

\begin{abstract}
This paper addresses the problem of event-triggered control for infinite-dimensional systems.
We employ event-triggering mechanisms that compare the plant state and
the error of the control input induced by the event-triggered implementation.
Under the assumption that feedback operators are compact,
a strictly positive lower bound on the inter-event times can be guaranteed.
We show that 
if the threshold of the event-triggering mechanisms is sufficiently small, then
the event-triggered control system with a bounded control operator and a compact feedback operator 
is exponentially stable.
For infinite-dimensional systems with unbounded control operators,
we employ two event-triggering mechanisms that are based on
system decomposition and periodic event-triggering, respectively, and then
analyze the exponential stability of the closed-loop system under each event-triggering
mechanism.
\end{abstract}



\section{Introduction}
The aim of this paper is to develop resource-aware control schemes for
infinite-dimensional systems.
To this end, we employ event-triggering mechanisms and analyze
exponential stability for infinite-dimensional event-triggered control systems.
Event-triggering mechanisms invoke data transmissions
if predefined conditions on the data are satisfied.
As a result, network and energy resources are consumed only when
the data is necessary for control.
In addition to such networked-control applications, 
the analysis and synthesis of event-triggered control systems are
interesting from a theoretical viewpoint, because the interaction of
continuous-time and discrete-time dynamics in event-triggered control systems is 
different from that in periodic sampled-data systems.
Most of the existing studies on event-triggered control 
have been developed for finite-dimensional systems, but
some researchers have recently extended to
infinite-dimensional systems, e.g., systems with output delays \cite{Lehmann2012},
first-order hyperbolic systems \cite{Espitia2016,Espitia2018, Prieur2017}, and
second-order parabolic systems \cite{Jiang2016, Selivanov2016PDE,Karafyllis2019}.
These earlier studies deal with specific delay differential equations and
partial differential equations. On the other hand,
the infinite-dimensional system we study is described by abstract evolution equations as follows.

For Hilbert spaces $X,U$, we here consider 
the following system with state space $X$
and input space $U$:
\begin{equation}
\label{eq:state_intro}
\dot x(t) = Ax(t) + Bu(t),\quad t \geq 0;\qquad x(0) = x^0 \in X,
\end{equation}
where $x(t) \in X$ is the state, $u(t) \in U$ is the input, $A$ is the generator of a strongly continuous semigroup $T(t)$
on $X$, and $B$ is a bounded linear operator from $U$ into the 
extrapolation space $X_{-1}$ associated with $T(t)$; see the notation and terminology section below
for
the definition of the 
extrapolation space $X_{-1}$.
If the control operator $B$ is bounded from $U$ to $X$,
$B$ is called {\em bounded}.
Otherwise, $B$ is called {\em unbounded}.
To illustrate the extrapolation space $X_{-1}$ and the unboundedness of the control operator $B$, 
consider the heat equation with 
Neumann boundary conditions, where $X = L^2([0,1],\mathbb{C})$ and $Ax := x''$
with domain $ D(A) = \{x \in W^{2,2}(0,1):
x'(0) = x'(1) = 0
\}$.
In this case,
$X_{-1}$ can be regarded as the dual of 
$D(A)$ (with the graph norm of $A$)
with respect to the pivot space $X$.
Therefore, $X_{-1}$ contains Dirac delta functions, which 
implies that for the case where $B$ is unbounded,
we can deal with point actuators.
In contrast, if $B$ is bounded, then
we consider only spatially-distributed actuators.
We refer the reader to Section~II.5 of \cite{Engel2000} and  Section~2.10 of \cite{Tucsnak2014}
for more details on the extrapolation space $X_{-1}$.

To infinite-dimensional systems described by the abstract 
evolution equation \eqref{eq:state_intro},
the results of periodic sampled-data control for 
finite-dimensional systems have been generalized in a number of papers; see
\cite{
	Rebarber1998,
	Rebarber2002, Logemann2003, Logemann2005, Rebarber2006, Logemann2013}.
Let an increasing sequence $\{t_k\}_{k\in \mathbb{N}_0}$ be the updating instants of the control input $u$, and consider
the feedback control
\[
u(t) = Fx(t_k),\quad t_k\leq t <  t_{k+1},~k \in \mathbb{N}_0,
\]
where $F$
is a bounded linear operator from $X$ to $U$.
In the standard sampled-data system, 
control updating is periodic, namely, $t_{k+1} - t_k $ is constant for every $k \in \mathbb{N}_0$. In contrast,
we employ the following event-triggering mechanisms in this paper:
\begin{align}
\label{eq:event_trigger_intro1}
t_0 &:= 0,\quad 
t_{k+1} := 
\inf \big\{t > t_{k}:\|Fx(t) - Fx(t_{k})\|_U > \varepsilon \|x(t_{k})\|_X
\big\} \quad \forall k \in \mathbb{N}_0 \\
\label{eq:event_trigger_intro2}
t_0 &:= 0,\quad 
t_{k+1} := 
\inf \big\{t > t_{k}:\|Fx(t) - Fx(t_{k})\|_U > \varepsilon \|x(t)\|_X
\big\} \quad \forall k \in \mathbb{N}_0,
\end{align}
where $\varepsilon>0$ is a threshold parameter.
If the threshold $\varepsilon$ is small, then 
the control input $u(t)$ is frequently updated. Therefore,
we would expect that the event-triggered control system
is exponentially stable for all sufficiently small thresholds $\varepsilon > 0$ if
$A+BF$ generates an exponentially stable semigroup.
One of the fundamental problems we consider is whether or not
this intuition is correct.

In addition to stability, the minimum inter-event time
$\inf_{k \in \mathbb{N}_0} (t_{k+1} - t_k)$ should be guaranteed to be 
bounded from below by a positive constant; otherwise, infinitely many events might occur in finite time.
This phenomenon is called Zeno behavior (see \cite{Goebel2009}) and
makes event-triggering mechanisms infeasible  for 
practical implementation.
The additional challenge of event-triggered control 
is to
guarantee no occurrence of Zeno behavior.
For finite-dimensional systems, the minimum inter-event time
has been extensively investigated; see \cite{Tabuada2007, Borgers2014,
	Dolk2017_Zeno}.
For example, in the finite-dimensional case, it has been shown in 
\cite{Tabuada2007}
that
the event-triggering mechanism
\[
t_0 := 0,\quad  t_k:=
\inf \big\{t > t_{k}:\|x(t) - x(t_{k})\|_X> \varepsilon \|x(t)\|_X\}\quad \forall k \in \mathbb{N}_0
\]
satisfies $\inf_{k \in \mathbb{N}_0} (t_{k+1} - t_k) \geq \theta$ for some $\theta >0$.
However, this is not true for the infinite-dimensional case, which is illustrated in
Examples~\ref{ex:zeno} and \ref{ex:minimum_inter_event_time_PPDE}. This is the reason why
we employ the event-triggering mechanisms \eqref{eq:event_trigger_intro1} and \eqref{eq:event_trigger_intro2}, 
which compare the plant state and the error $Fx(t) - Fx(t_k)$ of the control input
induced by the event-triggering implementation. 
We see in Section~3 that the time sequence $\{t_k\}_{k\in \mathbb{N}_0}$ defined by
\eqref{eq:event_trigger_intro1}  satisfies 
$\inf_{k \in \mathbb{N}_0} (t_{k+1} - t_k)  \geq \theta$ for some $\theta >0$ if the feedback operator $F$ is compact.
The same result holds for the event-triggering mechanism \eqref{eq:event_trigger_intro2} 
if a strictly positive lower bound on the decay of  $T(t)$ is guaranteed.

In the analysis of minimum inter-event times, we exploit the assumption on the compactness of
feedback operators, which may restrict applicability. In fact, as shown in Theorem~5.2.3 on p.~229 in \cite{Curtain1995},
if the generator $A$ has residual or continuous spectra in the closed right half plane, then
any compact feedback operator cannot guarantee the exponential stability of the semigroup generated by
$A+BF$.
However, if the unstable part of the system $(A,B)$ is finite-dimensional and controllable, 
then we can design a compact feedback operator that 
guarantees the exponential stability of the semigroup generated by
$A+BF$; see, e.g., Theorem~5.2.6 on p.~232 in \cite{Curtain1995}.
One particular example of infinite-dimensional systems with compact feedback operators 
is 
partial differential equations (PDEs) in cascade with ordinary differential equations (ODEs).
Stabilization of cascaded ODE-PDE systems have been recently studied,
e.g., in \cite{Susto2010,Andrade2018} and references therein. 
In the case where ODEs are located at the actuator side, 
the control input is applied to the finite-dimensional system
whose dynamics is described by the ODEs, and
the input space $U$ is finite-dimensional. Therefore,
the feedback operator has finite rank and hence is compact.


After guaranteeing that the minimum inter-event time is positive in Section 3, 
we analyze the exponential stability
of the event-triggered control system in Section~4 for
the case where the control operator $B$ is bounded.
First, we consider the event-triggering mechanism 
\eqref{eq:event_trigger_intro1} with the time constraint 
$t_{k+1} - t_k \leq \tau_{\max}$, $k \in \mathbb{N}_0$, where 
$\tau_{\max} >0$ can be chosen arbitrarily.
Introducing a norm on the state space with respect to which
the semigroup generated by $A+BF$ is a contraction, we
provide a sufficient condition on the threshold $\varepsilon$ of the event-triggering mechanism 
for the exponential stability of the closed-loop system.
Next we obtain a similar sufficient condition for the event-triggering mechanism \eqref{eq:event_trigger_intro2}.
While we obtain the former result via a trajectory-based approach, 
a key element in the latter result is the application of the Lyapunov stability theorem.

In Section 5, we study the case where $B$ is unbounded.
We first focus on a system with a finite-dimensional unstable part
and a feedback operator that stabilizes the unstable part but does not act on
the residual stable part.
In this case, the feedback operator has a specific structure, but
we can achieve the exponential stability of the closed-loop system
by using less conservative event-triggering mechanisms for the finite-dimensional part.
Second, 
we consider the case in which the semigroups $T(t)$ is analytic and
the feedback operator $F$ has no specific structure but is compact. 
Moreover, 
the semigroup generated by $A_{BF}$, where $A_{BF}x = (A+BF)x$ with domain $D(A_{BF}) := \{x \in X: (A+BF)x \in X\}$,
is assumed to be exponentially stable. Then we show that 
the closed-loop system is  exponentially stable under 
periodic event-triggering mechanisms \cite{Heemels2013, Heemels2013_Automatica}
with sufficiently small sampling periods and threshold parameters.

In Section~6, we illustrate the obtained results with
numerical examples.
First,  we study 
a cascaded system consisting of an ODE and a heat PDE as an example of
an infinite-dimensional system with a bounded control operator.
Second, we consider
an Euler-Bernoulli beam 
for the case where $B$ is unbounded.
From numerical simulations, we see that
the event-triggered control systems achieve
faster convergence with less control updates than
the conventional periodic sampled-data control systems.


\subsection*{Notation and terminology}
We denote by $\mathbb{Z}$ and $\mathbb{N}$ the set of  integers and
the set of positive integers, respectively.
Define 
$\mathbb{N}_0 := \mathbb{N} \cup \{0\}$, $\mathbb{R}_+ := [0,\infty)$, $
\mathbb{C}_- := \{s \in \mathbb{C} : \re s < 0\}
$, and
$\mathbb{Z}^* := \mathbb{Z} \setminus \{0\}$.
For $\alpha \in \mathbb{R}$, we define $\mathbb{C}_{\alpha} :=
\{s \in \mathbb{C} : \re s > \alpha \}$.
Let $X$ and $Y$ be Banach spaces. We denote
the space of all bounded linear operators from $X$ to $Y$ 
by $\mathcal{B}(X,Y)$, and set $\mathcal{B}(X) := \mathcal{B}(X,X)$.
We write $T^*$ for the adjoint operator of $T\in \mathcal{B}(X,Y)$.
Let $A$ be a linear operator from $X$ to $Y$. 
The domain of $A$ is denoted by $D(A)$.
For a subset $S \subset X$, let $A|_S$ denote
the restriction of $A$ to $S$, namely,
\[
A|_S x = Ax \qquad \forall x \in D(A) \cap S.
\]
The resolvent set and 
spectrum of a linear operator $A:D(A) \subset X \to X$
are denoted by $\varrho(A)$ and  $\sigma(A)$, respectively.
Let  $T(t)$ be
a strongly continuous semigroup on $X$.
The exponential growth bound of $T(t)$ is denoted by
$\omega(T)$, that is, $\omega(T) := \lim_{t \to \infty} \ln \|T(t)\|/t$.
We say that the strongly continuous semigroup $T(t)$ is {\em exponentially stable}
if $\omega(T) < 0$. 
The space $X_{-1}$ denotes the extrapolation space
associated with $T(t)$. 
More precisely,
if $A$ is the generator of $T(t)$, then
the space $X_{-1}$ is the completion of $X$ with
respect to the norm $\|\zeta \|_{-1} := \|(\lambda I - A)^{-1}\zeta \|$
for $\lambda \in \varrho(A)$.
Different choices of $\lambda\in \varrho(A)$ lead to equivalent norms on $X_{-1}$.
The semigroup $T(t)$ can be extended to a strongly continuous semigroup 
on $X_{-1}$, and its generator on $X_{-1}$ is an extension of $A$.
We shall use the same symbols $T(t)$ and $A$
for the original ones and the associated extensions.

\section{Infinite-dimensional system}
Let
an increasing sequence $\{t_k\}_{k\in \mathbb{N}_0}$ satisfy $t_0 = 0$ and
$
\inf_{k \in \mathbb{N}_0} (t_{k+1} - t_k)  > 0.
$
We denote by $X$ and $U$ the state space and the input space, 
and both of them are Hilbert spaces.
Let us denote by $\|\cdot\|$ and $\langle \cdot, \cdot \rangle$ 
the norm and the inner product of $X$, respectively.
As in  the periodic sampled-data case \cite{Logemann2003}, where $t_{k+1} - t_k$
is constant for every $k \in \mathbb{N}_0$,
consider the following infinite-dimensional system:
\begin{subequations}
	\label{eq:plant}
	\begin{align}
	\label{eq:state_equation}
	\dot x (t) &= Ax(t) +Bu(t),\quad  t \geq 0; \qquad 
	x(0) = x^0 \in X \\
	\label{eq:input}
	u(t) &= F x(t_k),\quad t_k\leq t <  t_{k+1},~k \in \mathbb{N}_0,
	\end{align}
\end{subequations}
where $x(t) \in X$ is the state and
$u(t) \in U$ is the input for $t\geq 0$.
We assume that 
$A$ is the generator of a strongly continuous semigroup $T(t)$ on $X$ and that the 
control operator $B$ and the feedback operator $F$ satisfy
$B \in \mathcal{B}(U,X_{-1})$ and $F \in  \mathcal{B}(X,U)$, respectively,
where $X_{-1}$ is the extrapolation space associated with $T(t)$.
We say that $B$ is {\em bounded} if $B \in \mathcal{B}(U,X)$; otherwise
$B$ is {\em unbounded}.
For example, if we control the temperature of a rod, then the control operator $B$ is bounded for
spatially-distributed actuators but is unbounded for point actuators; see also Chapters~3 and 4 of \cite{Curtain1995}.

The unique solution of the abstract evolution equation \eqref{eq:plant}
is given by
\begin{subequations}
	\label{eq:unique_solution}
	\begin{align}
	x(0) &= x^0 \\
	x(t_k+\tau) &= T(\tau) x(t_k) + \int^{\tau}_0 T(s)BFx(t_k) ds\quad 
	\forall \tau \in (0, t_{k+1}-t_k],~\forall k \in \mathbb{N}_0.
	\end{align}
\end{subequations}
In fact,
considering $T(t)$ as a semigroup on $X_{-1}$,
we find from the standard theory of strongly continuous semigroups
that $x$ given by \eqref{eq:unique_solution} satisfies
\begin{equation}
\label{eq:continuity}
x \in C(\mathbb{R}_{+}, X)\quad \text{and} \quad 
x|_{[t_k,t_{k+1}]} \in C^1([t_k,t_{k+1}], X_{-1})\quad \forall k \in \mathbb{N}_0
\end{equation}
and
the following differential equation in $X_{-1}$:
\begin{equation}
\label{eq:event_dynamics}
\dot x(t) = Ax(t) + BFx(t_k)\qquad \forall t \in (t_k,t_{k+1}),~\forall k \in \mathbb{N}_0.
\end{equation}
Moreover, only $x$ defined by \eqref{eq:unique_solution} satisfies
the properties \eqref{eq:continuity} and \eqref{eq:event_dynamics}.

\begin{definition}[Exponential stability]
	The system \eqref{eq:plant} is exponential stable  if 
	there exist $M \geq 1$ and $\gamma > 0$ such that 
	$x$ given by \eqref{eq:unique_solution} satisfies
	\[
	\|x(t)\| \leq M e^{-\gamma t} \|x^0\| \qquad 
	\forall x^0 \in X,~\forall t \geq 0. 
	\]
	The supremum over all possible values of $\gamma$ is called the 
	stability margin of the system \eqref{eq:plant}.
\end{definition}

\section{Minimum inter-event time}
We call  $\inf_{k \in \mathbb{N}_0} (t_{k+1} - t_k)$
the {\em minimum inter-event time}. If this value 
is zero, then  the control input may be updated infinitely fast, 
which
is not desirable for practical implementation.
The objective of this section is to show that 
the minimum inter-event time of
the following event-triggering mechanisms is 
bounded from below by a strictly positive constant:
\begin{align}
\label{eq:time_seq}
t_0 := 0,\quad 
&t_{k+1} := 
\inf \big\{t > t_{k}:\|Fx(t) - Fx(t_{k})\|_U > \varepsilon \|x(t_{k})\|
\big\} \quad \forall k \in \mathbb{N}_0 \\
\label{eq:time_seq_present}
t_0 := 0,\quad &t_{k+1} := 
\inf \big\{t > t_{k}:\|Fx(t) - Fx(t_{k})\|_U > \varepsilon \|x(t)\|
\big\} \quad \forall k \in \mathbb{N}_0.
\end{align}

In this and the next section, we employ the event-triggering mechanisms
\eqref{eq:time_seq} and 
\eqref{eq:time_seq_present}, which
compare the plant state and the error $Fx(t) - Fx(t_k)$ of the control input
induced by the event-triggering implementation. On the other hand, for finite-dimensional systems,
the  event-triggering mechanism
\begin{equation}
\label{eq:usual_EV1}
t_0 := 0,\quad 
t_{k+1} := 
\inf \big\{t > t_{k}:\|x(t) - x(t_{k})\| > \varepsilon \|x(t)\|\}
\quad \forall k \in \mathbb{N}_0
\end{equation}
is commonly used;
see, e.g., \cite{Tabuada2007}, in which it is proved that 
the minimum inter-event time of the 
event-triggering mechanism \eqref{eq:usual_EV1} is positive for finite-dimensional systems. 
However, in the infinite-dimensional case,
there exists a triple  of an infinite-dimensional systems, 
an initial state, and a feedback operator
such that
the inter-event time $t_{k+1} - t_k$ decreases to $0$ in finite time.
\begin{example}
	\label{ex:zeno}
	{\em 
		Let $X = L^2(0,\infty) := L^2([0,\infty), \mathbb{C})$ and consider the shift operator on
		$L^2(0,\infty)$:
		\[
		(T(t)x)(s) := x(t+s) \qquad \forall x \in L^2(0,\infty),~\forall s \geq 0.
		\]
		Then $T(t)$ is a strongly continuous semigroup on $L^2(0,\infty)$.
		As discussed in Remark~7 in \cite{Curtain1999}, 
		$T(t)$ is strongly stable but its adjoint $T(t)^*$ is not.
		
		Define an initial state $x^0 \in L^2(0,\infty)$ by
		\[
		x^0(s) :=
		\begin{cases}
		1 & s \leq 1 \\
		0  & s > 1,
		\end{cases}
		\]
		and $x(t) = T(t)x^0$ for all $t \geq 0$, which means that 
		the control operator $B$ is arbitrary but the feedback operator 
		$F$ satisfies $F=0$.

		For $\varepsilon \in(0,1)$,
		define a time sequence $\{t_k\}_{k\in \mathbb{N}_0}$ by
		\begin{equation}
		\label{eq:usual_EV2}
		t_0 := 0,\quad 
		t_{k+1} := 
		\inf \big\{t > t_{k}:\|x(t) - x(t_{k})\|_{L^2} > \varepsilon \|x(t_k)\|_{L^2}\}
		\quad \forall k \in \mathbb{N}_0.
		\end{equation}
		If $t_k \in [0,1)$, then $\|x(t_k)\|_{L^2}^2 = 1-t_k$ and
		\begin{align*}
		\|T(\tau)x(t_k) -  x(t_k)\|_{L^2}^2 =
		\int^{1 - t_k}_{1-(t_k+\tau)} 1 ds
		= \tau \qquad \forall \tau \in [0,1-t_k].
		\end{align*}
		It follows that
		\[
		t_{k+1} = t_k + \varepsilon^2(1-t_k) \qquad \forall k \in \mathbb{N}_0.
		\]
		Similarly, if we define $\{t_k\}_{k\in \mathbb{N}_0}$ by 
		the event-triggering mechanism  \eqref{eq:usual_EV1},
		then 
		\[
		t_{k+1} = t_k + \frac{\varepsilon^2}{1+\varepsilon^2} (1-t_k) \qquad \forall k \in \mathbb{N}_0.
		\]
		Both of the time sequences $\{t_k\}_{k\in \mathbb{N}_0}$ are monotonically increasing and 
		converge to 1. Thus 
		the minimum inter-event time $\inf_{k \in \mathbb{N}_0} (t_{k+1} - t_k)$ is zero.
		\hspace*{\fill} $\Box$
	}
\end{example}

We next provide an example of
infinite-dimensional event-triggered control systems in which $\inf_{x^0 \in X} t_1 = 0$.
\begin{example}
	\label{ex:minimum_inter_event_time_PPDE}
	{\em
		Consider a metal rod of length $1$ that 
		is insulated at either end and
		can be heated along its 
		length:
		\begin{subequations}
			\label{eq:PPDE}
			\begin{align}
			&\frac{\partial z}{\partial t}(\xi,t) = 
			\frac{\partial^2 z}{\partial \xi^2}(\xi,t) + v(\xi,t),\quad \xi \in [0,1],~t \geq 0 \\
			& \frac{\partial z}{\partial \xi}(0,t) = 0,\quad \frac{\partial z}{\partial \xi}(1,t) = 0,\quad t \geq 0,
			\end{align}
		\end{subequations}
		where $z(\xi,t)$ and $v(\xi,t)$ are the temperature of the rod and
		the addition of heat along the rod
		at position $\xi \in [0,1]$ and time $t\geq 0$, respectively.
		We can reformulate the partial differential equation \eqref{eq:PPDE} as 
		an abstract evolution equation \eqref{eq:state_equation}
		with $X := L^2(0,1):= L^2([0,1], \mathbb{C})$, $U := L^2(0,1)$, 
		$x(t) := z(\cdot,t)$, and $u(t) := v(\cdot,t)$. 
		As shown in 
		Example~2.3.7 on p.~45 in \cite{Curtain1995},
		the generator $A$, the strongly continuous semigroup $T(t)$ generated by $A$, and
		the control operator $B$
		are given by
		\begin{align*}
		A x := -\sum_{n = 0}^\infty 
		n^2\pi^2
		\langle
		x,\phi_n
		\rangle_{L^2} \phi_{n} 
		\end{align*}
		with domain
		\[
		D(A) := \left\{
		x \in L^2(0,1):
		\sum_{n = 0}^\infty  n^4\pi^4 
		|\langle
		x,\phi_n
		\rangle_{L^2}|^2 < \infty
		\right\}
		\]
		and
		\[
		T(t) x := 
		\sum_{n = 0}^\infty 
		e^{-n^2\pi^2 t}
		\langle
		x,\phi_n
		\rangle_{L^2} \phi_{n}\quad \forall x \in X,~\forall t \geq 0;\qquad B:=I,
		\]
		where $\phi_0(\xi) := 1$ and $\phi_n (\xi) := \sqrt{2} \cos(n \pi\xi )$, $n \in \mathbb{N}$, form
		an orthonormal basis for $L^2(0,1)$.
		Define the feedback operator $F \in \mathcal{B}(L^2(0,1))$ by
		$
		Fx := -
		\langle
		x,\phi_0
		\rangle_{L^2} \phi_0.
		$
		Although $T(t)$ is not exponentially stable, 
		the strongly continuous semigroup $T_{BF}(t)$ generated by $A+BF$ is exponentially stable.
		For this system, we consider the event-triggering 
		mechanism \eqref{eq:usual_EV1} and show $
		\inf_{x^0 \in X} t_1 = 0$ for every threshold $\varepsilon > 0$.
		
		Let the initial state $x^0$ be given by $x^0 := \phi_n$ with $n \in \mathbb{N}$. Then
		$Fx^0= 0$ and hence
		$
		x(t) = e^{-n^2 \pi^2 t} \phi_n
		$ for every $t \in [0,t_1)$.
		Since
		\[
		\|x(t) - x^0\| = 1-
		e^{-n^2 \pi^2 t} ,\quad
		\|x(t) \| = 
		e^{-n^2 \pi^2 t}\qquad \forall t \in [0,t_1)
		\]
		it follows that $t_1  = \log(1+\varepsilon)/ (n^2\pi^2)\to 0$ as $n \to \infty$.
		Thus, we obtain $\inf_{x^0 \in X} t_1 = 0$ for every $\varepsilon > 0$.
		\hspace*{\fill} $\Box$
	}
\end{example}

In Example~\ref{ex:zeno}, we consider a situation where
the state goes to zero in finite time with zero control input.
In Example~\ref{ex:minimum_inter_event_time_PPDE}, we only show that 
the first inter-event time is close to zero if we choose a certain initial state.
Therefore,
we cannot say that  the 
event-triggering mechanism \eqref{eq:usual_EV1} fails for practical control systems.
However,  these examples imply that 
there exists a triple of an infinite-dimensional system, an initial state, and
a compact feedback operator for which 
the minimum inter-event time can be arbitrarily close to zero. This is the reason
why we use the event-triggering mechanisms
\eqref{eq:time_seq} and 
\eqref{eq:time_seq_present}.

For $\tau \geq 0$,
define the operator $S_{\tau}: U \to X$ by
\[
S_{\tau}u := \int^{\tau}_0 T(s)Bu ds.
\]
The following lemma is useful when we evaluate
the minimum inter-event time:
\begin{lemma}[\hspace{0.01pt}Lemma 2.2 in \cite{Logemann2003}]
	\label{lem:Stau}
	For any $\tau\geq 0$, $S_{\tau} \in \mathcal{B}(U,X)$, and 
	for any $\theta >0$,
	\[
	\sup_{0\leq \tau\leq \theta} \|S_{\tau}\|_{\mathcal{B}(U,X)} < \infty.
	\]
	Moreover, if $F \in \mathcal{B}(X,U)$ is compact, then
	\begin{equation}
	\label{eq:SF_limit}
	\lim_{\tau \to 0} \|S_{\tau}F\|_{\mathcal{B}(X)} = 0.
	\end{equation}
\end{lemma}

For compact operators, the next lemma is also known:
\begin{lemma}[\hspace{0.01pt}Lemma 2.1 in \cite{Logemann2003}]
	\label{lem:uniform_conv}
	Let $X$, $Y$, and $Z$ be Hilbert spaces and let $\Gamma :[0,1] \to \mathcal{B}(X,Y)$
	be given. If $\lim_{t \to 0} \Gamma(t)^{*} y = 0$ for all $y \in Y$ and 
	if $\Lambda \in \mathcal{B}(Y,Z)$ is compact, then
	\[
	\lim_{t \to 0}  \|\Lambda \Gamma(t)\|_{\mathcal{B}(X,Z)} = 0.
	\]
\end{lemma}

Using these lemmas, we obtain the following result:
\begin{lemma}
	\label{thm:no_zeno}
	Assume that $T(t)$ is a strongly continuous semigroup on $X$,
	$B \in \mathcal{B}(U,X_{-1})$, and $F \in \mathcal{B}(X,U)$ is compact.
	Set
	\begin{equation}
	\label{eq:state_transition_for_zeno}
	x(\tau) = (T(\tau) +  S_{\tau}F)x^0\quad 
	\forall \tau > 0;\qquad x^0 \in X.
	\end{equation}
	For every $\varepsilon > 0$, 
	there exists $\theta >0$ such that
	\[
	\| Fx(\tau) - Fx^0\|_U\leq \varepsilon \|x^0\|\qquad 
	\forall x^0 \in X,~\forall \tau \in [0,\theta). \]
\end{lemma}
\begin{proof}
	Since 
	\eqref{eq:state_transition_for_zeno} yields
	\[
	x(\tau) - x^0 = (T(\tau) - I) x^0 + S_{\tau} Fx^0 \qquad \forall \tau \geq 0,
	\]
	it follows that 
	\begin{equation}
	\label{eq:Fx_bound}
	\| Fx(\tau) - Fx^0 \|_U \leq
	\|F (T(\tau) - I) + FS_{\tau }F\|_{\mathcal{B}(X,U)}   \cdot  \|x^0\|\qquad \forall \tau \geq 0.
	\end{equation}
	By Lemma~\ref{lem:Stau}, 
	\[
	\lim_{\tau \to 0 }\|FS_{\tau}F\|_{\mathcal{B}(X,U)} 
	= 
	0.
	\]
	Since $T(t)^*$ is strongly continuous (see, e.g., Theorem~2.2.6 on p.~37 in 
	\cite{Curtain1995}), we obtain
	\[
	\lim_{\tau \to 0 }
	\|(T(\tau)^* - I)x\| = 0 \qquad \forall x \in X,
	\]
	and hence 
	it follows from Lemma~\ref{lem:uniform_conv} that
	\[
	\lim_{\tau \to 0}\|F (T(\tau) - I)\|_{\mathcal{B}(X,U)}  = 0.
	\]
	Thus, for every $\varepsilon > 0$, 
	there exists $\theta >0$ such that
	for every $\tau \in [0,\theta)$,
	\begin{align}
	\|F (T(\tau) - I) + FS_{\tau }F\|_{\mathcal{B}(X,U)} 
	\leq 
	\|F (T(\tau) - I)\|_{\mathcal{B}(X,U)} + 
	\|FS_{\tau }F\|_{\mathcal{B}(X,U)}  
	< \varepsilon
	\label{eq:norm_bound_for_zeno}
	\end{align}
	Combining \eqref{eq:Fx_bound} and \eqref{eq:norm_bound_for_zeno},
	we obtain the desired result.
\end{proof}

From Lemma~\ref{thm:no_zeno},
we see that the minimum inter-event time of
the event-triggering mechanism \eqref{eq:time_seq}
is bounded from below by a strictly positive constant.
\begin{theorem}
	\label{coro:no_zeno}
	Assume that $A$ generates a strongly continuous semigroup $T(t)$ on $X$, $B \in \mathcal{B}(U,X_{-1})$, and $F \in \mathcal{B}(X,U)$ is compact.
	For the system \eqref{eq:plant},
	define the time sequence $\{t_k\}_{k\in \mathbb{N}_0}$ by the event-triggering mechanism
	\eqref{eq:time_seq}.
	For every  $\varepsilon > 0$, there exists $\theta >0$ such that 
	$\inf_{k \in \mathbb{N}_0} (t_{k+1} - t_k) \geq \theta$ for every initial state $x^0 \in X$.
\end{theorem}
\begin{proof}
	From \eqref{eq:continuity}, 
	it follows that $x(t_k) \in X$ for every $k \in \mathbb{N}_0$.
	Therefore, we see from \eqref{eq:unique_solution} and
	Lemma~\ref{thm:no_zeno} that, 
	for every $\varepsilon > 0$, 
	there exists $\theta > 0$ such that
	\[
	\| Fx(t_k+\tau) - Fx(t_k)\|_U\leq \varepsilon \|x(t_k)\|\qquad 
	\forall x^0 \in X,~\forall \tau \in [0,\theta),~\forall k \in \mathbb{N}_0. \]
	Thus, the
	time sequence $\{t_k\}_{k\in \mathbb{N}_0}$ satisfies
	$\inf_{k \in \mathbb{N}_0} (t_{k+1} - t_k) \geq \theta.$
\end{proof}

We next investigate 
the minimum inter-event time of 
the event-triggering mechanism \eqref{eq:time_seq_present}.
To that purpose,
we use the following estimate:
\begin{lemma}
	\label{lem:no_zeno_xt_bound}
	Assume that $T(t)$ is a strongly continuous semigroup on $X$,
	$B \in \mathcal{B}(U,X_{-1})$, and $F \in \mathcal{B}(X,U)$ is compact.
	Define $x(t)$ as in \eqref{eq:state_transition_for_zeno}.
	There exist $c_1 > 0$  and $s_1 > 0$ 
	such that the semigroup $T(t)$ satisfies
	\begin{align}
	\|T(s_1)x^0\| \geq c_1 \|x^0\|  \qquad \forall x^0 \in X  \label{eq:T1_coersive} 
	\end{align}
	if and only if there exist $c_2 \geq 1$ and $\theta > 0$ such that
	\begin{align}
	\label{eq:x0_bounded_by_x_tau}
	\|x^0\| \leq c_2 \|x(\tau)\|\qquad  \forall x^0 \in X,~\forall \tau \in [0,\theta).
	\end{align}
\end{lemma}

\begin{proof}
	Suppose first that \eqref{eq:T1_coersive}  holds for some $c_1 > 0$ and $s_1 >0$.
	Since there exists $M \geq 1$ such that 
	\[
	\|T(\tau)\|_{\mathcal{B}(X)}  \leq M \qquad \forall \tau \in [0,s_1],
	\]
	it follows from \eqref{eq:T1_coersive} that 
	\[
	c_1 \|x^0\| \leq \|T(s_1)x^0\| = \|T(s_1 - \tau) T(\tau)x^0 \|
	\leq M\| T(\tau)x^0  \| \qquad \forall x^0 \in X,~\forall \tau \in [0,s_1].
	\]
	Therefore,
	\begin{equation}
	\label{eq:T1_coersive_all}
	\|T(\tau)x^0\| \geq 
	\frac{c_1}{M} \|x^0\| \qquad \forall x^0 \in X,~\forall \tau \in [0,s_1].
	\end{equation}
	By \eqref{eq:SF_limit}, 
	there exists $s_2 \in (0,s_1]$ such that 
	\begin{equation}
	\label{eq:StF_bound}
	\|S_\tau Fx^0\| \leq \|S_\tau F\|_{\mathcal{B}(X)}  \cdot \|x^0\|  \leq \frac{c_1}{2M} \|x^0\|\qquad
	\forall x^0 \in X,~ \forall \tau \in [0,s_2].
	\end{equation}
	Combining \eqref{eq:T1_coersive_all} and \eqref{eq:StF_bound},
	we obtain
	\begin{align}
	\label{eq:TSF_coersive}
	\|x(\tau)\|
	\geq 
	\|T(\tau)x^0\| - \|S_\tau Fx^0\|
	\geq  
	\frac{c_1}{2M} \|x^0\| \qquad 
	\forall x^0 \in X,~\forall \tau \in [0,s_2].
	\end{align}
	Therefore, \eqref{eq:x0_bounded_by_x_tau} holds
	with $c_2 := 2M/c_1$ and $\theta := s_2$.
	
	Conversely, if \eqref{eq:x0_bounded_by_x_tau} holds
	for some $c_2 \geq 1$ and  $\theta >0$, then
	\[
	\|x^0\| \leq
	c_2 \|T(\tau) x^0\| + c_2\|S_\tau Fx^0\|\qquad 
	\forall x^0 \in X,~\forall \tau \in [0,\theta).
	\]
	Using \eqref{eq:SF_limit} again,
	we find that 
	there exists $s_1 \in (0,\theta)$ such that
	\[
	\|S_\tau F\|_{\mathcal{B}(X)}  \leq \frac{1}{2c_2}\qquad \forall \tau \in [0,s_1].
	\]
	Hence,
	\[
	\|T(\tau) x^0\| 
	\geq \frac{1}{c_2} \left(
	\|x^0\| - c_2 \|S_{\tau} F x^0\|
	\right)
	\geq \frac{1}{2c_2} \|x^0\| \qquad \forall
	x^0 \in X,~\forall \tau \in [0,s_1].
	\]
	Thus the desired inequality \eqref{eq:T1_coersive}  holds.
\end{proof}
%

\begin{remark}
	{\em
		Suppose that  $B$ is bounded, i.e., $B \in \mathcal{B}(U,X)$. Then
		\[\lim_{\tau \to 0} \|S_{\tau}\|_{\mathcal{B}(U,X)}  = 0,\] and hence 
		the compactness of $F$ is not required in Lemma~\ref{lem:no_zeno_xt_bound}.
	}
\end{remark}

\begin{remark}
	{\em
		The condition \eqref{eq:T1_coersive} appears also in   
		Theorem 2 of
		\cite{Pazy1972} for 
		the applicability of the Lyapunov stability theorem, and
		Corollary 1 of
		\cite{Pazy1972} shows that
		$T(t)$ satisfies \eqref{eq:T1_coersive} for some $c_1 >0$ and $s_1 > 0$ and
		the range of $T(t)$ is dense in $X$ for some $t \in (0,s_1]$ if and only if
		$T(t)$ can be extended to a strongly continuous group on $X$.
	}
\end{remark}

Using Lemmas~\ref{thm:no_zeno} and \ref{lem:no_zeno_xt_bound}, we show that 
the minimum inter-event time of
the event-triggering mechanism \eqref{eq:time_seq_present} 
is bounded from below by a strictly positive constant
if $T(t)$ satisfies \eqref{eq:T1_coersive}
for some $c_1 >0$ and $s_1 > 0$.
\begin{theorem}
	\label{thm:no_zeno_present}
	Assume that $A$ generates a strongly continuous semigroup $T(t)$ on $X$, $B \in \mathcal{B}(U,X_{-1})$, and $F \in \mathcal{B}(X,U)$ is compact.
	Assume further that
	the semigroup $T(t)$ satisfies \eqref{eq:T1_coersive} for some $c_1 > 0$ and $s_1 >0$.
	For the system \eqref{eq:plant},
	define the time sequence $\{t_k\}_{k\in \mathbb{N}_0}$ by
	the event-triggering mechanism
	\eqref{eq:time_seq_present}.
	For every  $\varepsilon > 0$, there exists $\theta >0$ such that 
	$\inf_{k \in \mathbb{N}_0} (t_{k+1} - t_k) \geq \theta$ for every initial state $x^0 \in X$.
\end{theorem}
\begin{proof}
	Combining Lemmas~\ref{thm:no_zeno} and \ref{lem:no_zeno_xt_bound},
	we have that, 
	for every $\varepsilon > 0$, 
	there exists $\theta > 0$ such that
	\[
	\| Fx(t_k+\tau) - Fx(t_k)\|_U \leq \varepsilon \|x(t_k+\tau)\|\qquad 
	\forall x^0 \in X,~\forall \tau \in [0,\theta),~ \forall k \in \mathbb{N}_0. \]
	Thus the
	time sequence $\{t_k\}_{k\in \mathbb{N}_0}$ satisfies
	$\inf_{k \in \mathbb{N}_0} (t_{k+1} - t_k) \geq \theta$.
\end{proof}

We conclude this section with a result on the
continuous dependence of solutions of the evolution equation \eqref{eq:plant}  under
the event-triggering mechanism \eqref{eq:time_seq_present} on initial states.
The technical difficulty  is that
the updating instants of control inputs are different
depending on initial states in event-triggered control systems.
The analysis of continuous dependence on initial states
is straightforward but lengthy. The proof of the following theorem
can be found in the Appendix~\ref{sec:appendixA}.
\begin{theorem}
	\label{thm:continuity_initial_state}
	Assume that $A$ generates a strongly continuous semigroup $T(t)$ 
	on $X$, $B \in \mathcal{B}(U,X_{-1})$, and $F \in \mathcal{B}(X,U)$ is compact.
	Assume further that 
	the semigroup $T(t)$ satisfies \eqref{eq:T1_coersive} for 
	some $c_1 > 0$ and $s_1 >0$.
	Let $x$ be the solution of
	the evolution equation \eqref{eq:plant} with the initial state $x^0 \in X$ 	under
	the 
	event-triggering mechanism \eqref{eq:time_seq_present} with an arbitrary threshold $\varepsilon >0$.
	For every $t_{\rm e} >0$ and every $\delta >0$, there exists $\delta_0 >0$ such that 
	for every $\zeta^0 \in X$ satisfying $\|x^0 - \zeta^0\| < \delta_0$,
	\[
	\|x(t) - \zeta(t)\| < \delta \qquad \forall t \in [0,t_{\rm e}],
	\]
	where $\zeta$ is the solution of the evolution equation \eqref{eq:plant}
	with the initial state $\zeta^0$ under 	
	the 
	event-triggering mechanism \eqref{eq:time_seq_present}.
\end{theorem}

\section{Stability analysis under bounded control}
In this section, we analyze closed-loop stability in the case
where the control operator $B$ is bounded.
The objective is to show that if
the feedback operator $F$ is compact and if 
the semigroup generated by $A+BF$ is exponentially stable, then
the event-triggered control system is exponentially stable, provided that
the threshold $\varepsilon$ is sufficiently small.

Choose $\tau_{\max} > 0$ arbitrarily.
We first define a time sequence
$\{t_k\}_{k\in \mathbb{N}_0}$ by
\begin{subequations}
	\label{eq:time_seq_upper_bound}
	\begin{align}
	t_0 &:=0,\quad \psi_{k+1} :=\inf \big\{t > t_{k}:\|Fx(t) - Fx(t_{k})\|_U > \varepsilon \|x(t_{k})\|
	\big\} \\
	t_{k+1} &:= 
	\min\{ t_k + \tau_{\max},~\psi_{k+1}\}
	\qquad \forall k\in \mathbb{N}_0.
	\end{align}
\end{subequations}
The above event-triggering mechanism is based on \eqref{eq:time_seq} and satisfies
the time constraint $t_{k+1} - t_k \leq \tau_{\max}$ for every $k \in \mathbb{N}_0$.
\begin{theorem}
	\label{thm:bounded1}
	Assume that $A$ generates a strongly continuous semigroup $T(t)$ on $X$, $B \in \mathcal{B}(U,X)$, and
	$F \in \mathcal{B}(X,U)$ is compact.
	Assume that the semigroup $T_{BF}(t)$ generated by $A+BF$
	is exponentially stable, i.e.,
	there exists $M \geq 1$ and $\omega >0$ such that
	\begin{equation}
	\label{eq:TBF_exp_stability}
	\|T_{BF}(t)\|_{\mathcal{B}(X)} \leq M e^{-\omega t}\qquad \forall t \geq 0.
	\end{equation}
	If the threshold
	$\varepsilon >0$ satisfies
	\begin{equation}
	\label{eq:ep_cond}
	\varepsilon < \frac{\omega}{M \|B\|_{\mathcal{B}(U,X)} },
	\end{equation}
	then
	for every $\tau_{\max} >0$,
	the system \eqref{eq:plant} with  the event-triggering mechanism  \eqref{eq:time_seq_upper_bound}
	is exponentially stable and
	its stability margin is at least
	$\gamma$ defined by
	\begin{equation}
	\label{eq:bounded_decay_rate}
	\gamma := 
	\frac{-\log \big((1-\varepsilon_0)e^{-\omega \tau_{\max}} + \varepsilon_0 \big)}{\tau_{\max}},
	\text{~~where~~}	\varepsilon_0 := \varepsilon
	\frac{M \|B\|_{\mathcal{B}(U,X)}  }{\omega}.
	\end{equation}		
\end{theorem}
\begin{proof}
	As in the proof of
	Theorem 5.2 on p.~19 in \cite{Pazy1983} and
	Theorem 3.1 in \cite{Logemann2003}, we introduce a new norm $|\cdot |$ on $X$, which is 
	defined by
	\[
	|x| := \sup_{t \geq 0} \|e^{\omega t} T_{BF}(t)x\|.
	\]
	This norm satisfies
	\begin{equation}
	\label{eq:new_norm_prop}
	\|x\| \leq |x| \leq M \|x\|,\quad
	|T_{BF}(t)x| \leq e^{-\omega t} |x|
	\qquad \forall x\in X,~ \forall t \geq 0;
	\end{equation}
	see, e.g.,  the proof of Theorem 3.1 in \cite{Logemann2003}.

	Noting that
	\begin{align}
	\dot x (t) 
	= (A+BF) x(t) - B\big(Fx(t) - Fx(t_k)\big)\qquad \forall t \in (t_k,t_{k+1}),~\forall k \in \mathbb{N}_0, \label{eq:diff_eq_ABF}
	\end{align}
	we see from a routine calculation (see, e.g., Exercise 3.3 in \cite{Curtain1995}) that
	$x(t_k+\tau)$ given in \eqref{eq:unique_solution} can be written as
	\begin{equation*}
	x(t_k+\tau) = T_{BF}(\tau) x(t_k) - \int^\tau_0 T_{BF}(\tau-s) B\big(Fx(t_k+s) - Fx(t_k)\big) ds
	\end{equation*}
	for every $\tau \in (0,t_{k+1}-t_k]$.
	Since the event-triggering mechanism \eqref{eq:time_seq_upper_bound} 
	guarantees
	\[
	\sup_{0\leq s < t_{k+1} -t_k }\|Fx(t_k+s) - Fx(t_k) \|_U \leq \varepsilon \|x(t_k)\|
	\qquad \forall k \in \mathbb{N}_0,
	\]
	the properties \eqref{eq:new_norm_prop} yield
	\begin{align*}
	|x(t_{k}+\tau)| 
	&\leq e^{-\omega \tau} |x(t_k)| +  \int^{\tau}_0e^{-\omega (\tau-s)} \big|B\big(Fx(t_k+s) - Fx(t_k)\big) \big| ds \\
	&\leq e^{-\omega \tau} |x(t_k)| +  \frac{1 - e^{-\omega \tau}}{\omega} M \|B\|_{\mathcal{B}(U,X)}   \cdot \sup_{0\leq s < \tau }\|Fx(t_k+s) - Fx(t_k) \|_U \\
	&\leq \left( (1-\varepsilon_0)e^{-\omega \tau} + \varepsilon_0 \right) \cdot |x(t_k)|
	\qquad  \forall \tau \in  (0,t_{k+1}-t_k],~ \forall k\in \mathbb{N}_0,
	\end{align*}
	where $\varepsilon_0$ is defined as in \eqref{eq:bounded_decay_rate}
	and satisfies $0< \varepsilon_0 < 1$ from \eqref{eq:ep_cond}.
	
	The function $f:(0,\infty) \to \mathbb{R}$ defined by
	\[
	f(\tau) :=
	\frac{-\log \big((1-\varepsilon_0)e^{-\omega \tau} + \varepsilon_0 \big)}{\tau}
	\]
	is monotonically decreasing in $(0,\infty)$. Since 
	\[
	(1-\varepsilon_0)e^{-\omega \tau} + \varepsilon_0 < 
	(1-\varepsilon_0) + \varepsilon_0= 1\qquad
	\forall \tau > 0,
	\]
	it follows that  $f(\tau) > 0$ for every $\tau>0$.
	Therefore, $\gamma := f(\tau_{\max})$
	satisfies $\gamma > 0$ and
	\begin{align}
	\label{eq:x_k1_k}
	|x(t_{k} + \tau)| \leq e^{-\gamma \tau} |x(t_k)|\qquad \forall \tau \in  (0,t_{k+1}-t_k],~ \forall k\in \mathbb{N}_0.
	\end{align}
	Using \eqref{eq:x_k1_k} recursively, we obtain
	\[
	|x(t)| \leq e^{-\gamma t} |x^0|\qquad  \forall x^0 \in X,~\forall t \geq0.
	\]
	Thus, 
	\eqref{eq:new_norm_prop} yields
	\[
	\|x(t)\| \leq |x(t)| \leq e^{-\gamma t} |x^0| \leq M e^{-\gamma t} \|x^0\|\qquad 
	\forall x^0 \in X,~\forall t \geq0.
	\]
	This completes the proof.
\end{proof}

Next we define the time sequence
$\{t_k\}_{k\in \mathbb{N}_0}$ by
the event-triggering mechanism \eqref{eq:time_seq_present}
and show the exponential stability of the closed-loop system.
Instead of the trajectory-based approach in Theorem~\ref{thm:bounded1},
we here apply the Lyapunov stability theorem.
\begin{theorem}
	\label{thm:bounded2}
	Assume that $A$ generates a strongly continuous semigroup $T(t)$ on $X$, $B \in \mathcal{B}(U,X)$, and 
	$F \in \mathcal{B}(X,U)$ is compact.
	Assume further that
	the semigroup $T(t)$ satisfies \eqref{eq:T1_coersive} for some $c_1 > 0$ and $s_1 >0$
	and that
	the semigroup $T_{BF}(t)$ generated by $A+BF$ 
	is exponentially stable.
	If  the threshold
	$\varepsilon >0$ satisfies
	\begin{equation}
	\label{eq:ep_cond2}
	\varepsilon < \frac{1}{2 \|PB\|_{\mathcal{B}(U,X)} },
	\text{~~where~~} P x:=
	\int^{\infty}_0 T_{BF}(t)^* T_{BF}(t)x dt\quad \forall x \in X,
	\end{equation}
	then the 
	system \eqref{eq:plant} with the event-triggering mechanism \eqref{eq:time_seq_present}
	is exponentially stable 
	and
	its stability margin is at least
	$\gamma$ defined by
	\begin{equation}
	\label{eq:bounded_decay_rate2}
	\gamma := 
	\frac{1-2\varepsilon\|PB\|_{\mathcal{B}(U,X)}}{2\int_0^{\infty} \|T_{BF}(t)\|_{\mathcal{B}(X)}^2 dt}.
	\end{equation}
\end{theorem}
\begin{proof}
	There exists $M \geq 1$ such that
	\begin{equation}
	\label{eq:T_Tbf_bound}
	\|T(\tau)\|_{\mathcal{B}(X)}  \leq M,\quad \|T_{BF}(\tau)\|_{\mathcal{B}(X)} 
	\leq M \qquad \forall \tau \in [0,s_1].
	\end{equation}
	Using \eqref{eq:T1_coersive},
	we obtain
	\[
	c_1 \|x\| \leq \|T(s_1)x\| = \|T(s_1 - \tau) T(\tau)x \|
	\leq M\| T(\tau)x  \| \qquad \forall x \in X,~\forall \tau \in [0,s_1].
	\]
	Hence
	\begin{equation}
	\label{eq:T1_coersive_all2}
	\|T(\tau)x\| \geq 
	\frac{c_1}{M} \|x\| \qquad \forall x \in X,~\forall \tau \in [0,s_1].
	\end{equation}
	Since $T_{BF}(\tau)$ satisfies the variation of
	parameters formula  (see Theorem 3.2.1  on p.~110
	in 
	\cite{Curtain1995}):
	\begin{equation}
	\label{eq:T_BF_identity}
	T_{BF}(\tau) x = T(\tau)x  + 
	\int^{\tau}_0 T_{BF}(\tau-s) BF
	T(s) xds\qquad \forall x \in X,
	\end{equation}
	it follows from 
	\eqref{eq:T_Tbf_bound} and \eqref{eq:T1_coersive_all2}
	that
	\begin{align*}
	\|T_{BF}(\tau) x\| 
	&\geq \|T(\tau)x\|  -
	\int^{\tau}_0 \|T_{BF}(\tau-s) BF
	T(s) x\| ds \\
	&\geq \left( \frac{c_1}{M} -
	\tau M^2\|BF\|_{\mathcal{B}(X)}  \right) \|x\|\qquad 
	\forall  x \in X,~\forall \tau \in [0,s_1].
	\end{align*}
	Therefore, 
	\begin{equation}
	\label{eq:Tbf_coercive}
	\|T_{BF}(\tau) x\|  \geq \frac{c_1}{2M} \|x\|
	\qquad \forall x \in X,~\forall \tau \in [0,s_3],
	\end{equation}
	where
	\[
	s_3 := \min
	\left\{
	s_1,~
	\frac{c_1}{2M^3 \|BF\|_{\mathcal{B}(X)} }
	\right\} > 0.
	\]
	
	Since $T_{BF}(t)$ is exponentially stable, it follows
	from Theorem 5.1.3 on p.~217 in \cite{Curtain1995} that
	there exists a positive operator $P \in \mathcal{B}(X)$ such that 
	the following Lyapunov equality holds:
	\begin{equation}
	\label{eq:Lyapunov_eq}
	\langle
	(A+BF)x,Px
	\rangle
	+
	\langle
	Px, (A+BF)x
	\rangle
	=-\|x\|^2
	\qquad \forall x \in D(A),
	\end{equation}
	and such an operator $P$ is given as in \eqref{eq:ep_cond2}.
	Using \eqref{eq:Tbf_coercive},
	we have from Theorem 2 in \cite{Pazy1972} that
	there exist $\alpha,\beta > 0$ such that
	\begin{equation}
	\label{eq:positive_ness}
	\alpha \|x\|^2 \leq \langle Px, x \rangle \leq \beta \|x\|^2 \qquad \forall x \in X.
	\end{equation}
	Since 
	\[
	\langle Px, x \rangle =
	\int^{\infty}_0 
	\|T_{BF}(t) x\|^2  dt
	\leq 
	\int^{\infty}_0 
	\|T_{BF}(t)\|^2  dt \cdot \|x\|^2,
	\]
	we can choose $\beta > 0$ in \eqref{eq:positive_ness} so that 
	\begin{equation*}
	\beta \leq \int^{\infty}_0 
	\|T_{BF}(t)\|^2  dt.
	\end{equation*}
	
	Assume that $x^0 \in D(A)$.
	Then the mild solution $x$ of \eqref{eq:unique_solution} is also a classical solution,
	i.e., $x$ satisfies the differential equation \eqref{eq:event_dynamics} in $X$;
	see. e.g., Theorem~3.1.3 on p.~103 in \cite{Curtain1995}.
	Moreover, if we define
	$e(t) := Fx(t) - F(t_k)$ for $t_k \leq t < t_{k+1}$ and $k \in \mathbb{N}_0$, which is
	the  error 
	induced by the event-triggered implementation,
	then 
	\[
	\|e(t)\|_U \leq \varepsilon \|x(t)\| \qquad \forall t\geq 0
	\]
	under the event-triggering mechanism \eqref{eq:time_seq_present}.
	Therefore,  
	using \eqref{eq:diff_eq_ABF} and \eqref{eq:Lyapunov_eq}, we find that
	for every $t \in (t_k,t_{k+1})$ and every $k \in \mathbb{N}_0$,
	$V(t) := \langle Px(t), x(t) \rangle$ satisfies
	\begin{align*}
	\frac{dV}{dt}(t) 
	&\leq
	-\|x(t)\|^2 + 
	2\|PB\|_{\mathcal{B}(U,X)}  \cdot \|x(t)\| \cdot \|e(t)\|_U \\
	&\leq
	-(1-2\varepsilon\|PB\|_{\mathcal{B}(U,X)} ) \|x(t)\|^2 \\
	&\leq
	-2\gamma
	V(t),
	\end{align*}
	where $\gamma >0$ is defined as in \eqref{eq:bounded_decay_rate2}.
	We see from the positive definiteness \eqref{eq:positive_ness} of $P$ that
	\[
	\alpha \|x(t)\|^2 \leq V(t) \leq e^{-2\gamma t} V(0) \leq 
	\beta e^{-2\gamma t} \|x^0\|^2\qquad \forall t \geq 0,
	\]
	and hence
	\begin{equation}
	\label{eq:x_bound_DA}
	\|x(t)\| \leq \sqrt{\frac{\beta}{\alpha}} e^{-\gamma t} \|x^0\|
	\qquad \forall x^0 \in D(A),~ \forall t \geq 0.
	\end{equation}
	
	Finally, we show the exponential stability for all initial states in $X$.
	Fix $x^0 \in X$ and $t_{\rm e} \geq 0$ arbitrarily, and let $x$ be
	the solution of the  evolution equation 
	\eqref{eq:unique_solution} with the initial state $x^0$.
	Since $D(A)$ is dense in $X$, Theorem~\ref{thm:continuity_initial_state} shows that 
	for every $\delta> 0$,
	there exists $\zeta^0 \in D(A)$ such that 
	the solution $\zeta$ of the evolution equation 
	\eqref{eq:unique_solution} with the initial state $\zeta^0$ satisfies
	\[
	\|x(t) - \zeta(t)\| < \delta \qquad \forall t \in [0,t_{\rm e}].
	\]
	Therefore we have from \eqref{eq:x_bound_DA} that 
	\begin{align*}
	\|x(t)\| \leq \|x(t) - \zeta(t)\| + \|\zeta(t)\|
	< \left( 1 + \sqrt{\frac{\beta}{\alpha}} \right) \delta + 
	\sqrt{\frac{\beta}{\alpha}} e^{-\gamma t} \|x^0\| \qquad \forall t \in [0,t_{\rm e}].
	\end{align*}
	Since $\delta >0$ was arbitrary, we obtain 
	\begin{align*}
	\|x(t)\| \leq
	\sqrt{\frac{\beta}{\alpha}} e^{-\gamma t} \|x^0\| \qquad \forall t \in [0,t_{\rm e}].
	\end{align*}
	Moreover, $t_{\rm e} \geq 0$ was also arbitrary. Thus, the closed-loop
	system \eqref{eq:plant} is exponentially stable and
	its stability margin is at least $\gamma$.
	%
	%
	
\end{proof}

\begin{remark}
	{\em
		If 
		$M \geq 1$ and $\omega > 0$ satisfy
		$\|T_{BF}(t)\|_{\mathcal{B}(X)}  \leq M e^{-\omega t}$ for all $t \geq 0$,
		then we obtain
		$\|P\|_{\mathcal{B}(X)}  \leq M^2/(2\omega).$
		Therefore, 
		\eqref{eq:ep_cond2} and 
		\eqref{eq:bounded_decay_rate2} can be rewritten as
		\begin{align*}
		\varepsilon < \frac{\omega}{M^2 \|B\|_{\mathcal{B}(U,X)} },\qquad
		\gamma \geq \frac{\omega  - \varepsilon M^2\|B\|_{\mathcal{B}(U,X)}  }{M^2}.
		\end{align*}
	}
\end{remark}

\section{Stability analysis under unbounded control}
Throughout this section, we study the case 
$B \in \mathcal{B}(U,X_{-1})$.
We analyze the exponential stability of the closed-loop system under
two event-triggering mechanisms.
The first mechanism is based on system decomposition,
and the second one employs a periodic event-triggering condition developed in 
\cite{Heemels2013, Heemels2013_Automatica}.
\subsection{Event-triggered control based on system decomposition}
\label{subsec:decomp}
\subsubsection{System decomposition}
In what follows, we 
shall place a number of assumptions on the infinite-dimensional system
\eqref{eq:plant} and
recall the decomposition
of infinite-dimensional systems under unbounded control used in 
\cite{Logemann2003, Logemann2005, Logemann2013}.

\begin{assumption}
	\label{assump:finite_multiplicities}
	There exists $\alpha < 0$ such that 
	$\sigma (A) \cap \overline{\mathbb{C}}_{\alpha}$ consists of finitely many
	eigenvalues of $A$ with finite algebraic multiplicities.
\end{assumption}
If Assumption~\ref{assump:finite_multiplicities} holds, then
we can decompose $X$ by a standard technique (see, e.g., Lemma 2.5.7 on p.~71 in \cite{Curtain1995}
or Proposition~IV.1.16 on p.~245 in \cite{Engel2000})
as follows.
There exists a rectifiable, closed, simple curve $\Gamma$ in $\mathbb{C}$ enclosing an open set that 
contains 
$\sigma(A) \cap \overline{\mathbb{C}}_\alpha$ in its interior and
$\sigma(A) \cap (\mathbb{C} \setminus \overline{\mathbb{C}}_\alpha)$ 
in its exterior.
The operator $\Pi: X \to X$, defined by
\begin{equation}
\label{eq:Pi_def}
\Pi := \frac{1}{2\pi i} \int_{\Gamma} (sI-A)^{-1}ds,
\end{equation}
where $\Gamma$ is traversed once in the counterclockwise direction, 
is a projection operator, and 
we can decompose $X$ to be
\begin{equation}
\label{eq:X_decomposition}
X = X^+ \oplus X^-,
\text{~~where~~$X^+ := \Pi X$~~and~~$X^- := (I-\Pi) X$. }
\end{equation}
This decomposition satisfies
$\dim X^+ < \infty$ and $X^+ \subset D(A)$. Moreover,
$X^+$ and $X^-$ are $T(t)$-invariant for all $t \geq 0$.
Define 
\begin{equation}
\label{eq:Aj_def}
A^+ := A|_{X^+}, \quad
A^- := A|_{X^-},\qquad 
T^+(t) := T(t)|_{X^+},\quad 
T^-(t) := T(t)|_{X^-}.
\end{equation}
Then 
$
\sigma(A^+) = \sigma(A) \cap \overline{\mathbb{C}}_\alpha
$ and 
$\sigma(A^-) = \sigma(A) \cap (\mathbb{C} \setminus \overline{\mathbb{C}}_\alpha).
$
Note that $A^+$ and $A^-$ generate the strongly continuous semigroups $T^+(t)$ on $X^+$ and
$T^-(t)$ on $X^-$, respectively.
Let $(X^-)_{-1}$ denote the extrapolation space associated with $T^{-}(t)$.
The semigroup $T^-(t)$ can be extended to a strongly continuous semigroup
on $(X^-)_{-1}$, and 
its generator on $(X^-)_{-1}$ is an extension of $A^-$ on $X^-$.
The same symbols $T^-(t)$ and $A^-$ will be used to denote these extensions.
Since the spectrum of the operator $A$ on $X$ is equal to 
the spectrum of the operator $A$ on $X_{-1}$,
the projection operator $\Pi$ on $X$ defined by \eqref{eq:Pi_def}
can be extended to a projection $\Pi_{-1}$ on $X_{-1}$.
If $\lambda \in \varrho(A)$ and if $\lambda I- A$ is considered as
an  operator in $\mathcal{B}(X,X_{-1})$, then $\Pi_{-1}$ is similar to $\Pi$, i.e.,
\[
\Pi_{-1} = (\lambda I - A) \Pi (\lambda I - A)^{-1}
\]
and satisfies $\Pi_{-1}X_{-1} = \Pi X = X^+$.
Using the extended projection operator $\Pi_{-1}$, we can decompose
the control operator $B \in \mathcal{B}(U,X_{-1})$:
\[
B^+ := \Pi_{-1} B,\quad B^- := (I - \Pi_{-1})B.
\]
Since $(X^-)_{-1}$ and $(X_{-1})^- := (I - \Pi_{-1})X_{-1}$ are both completions of $X^-$ 
endowed with the norm $\|\cdot \|_{-1}$, 
we can identify $(X^-)_{-1}$ and $(X_{-1})^-$ (see, e.g., the footnote 3 of p. 1213 in 
\cite{Logemann2013}).
We also decompose the feedback operator $F \in \mathcal{B}(X,U)$:
\[
F^+ := F|_{X^+},\quad F^- := F|_{X^-}.
\]

In addition to Assumption~\ref{assump:finite_multiplicities},
we impose the following assumptions:
\begin{assumption}
	\label{assump:exponential_stability_T-}
	The exponential growth bound $\omega(T^-)$ satisfies $\omega(T^-) < 0$.
\end{assumption}
\begin{assumption}
	\label{assump:controllability}
	The pair $(A^+, B^+)$ is controllable.
\end{assumption}

\begin{remark}
	\label{rem:A_analytic}
	{\em
		It is shown in \cite{Logemann2003} that
		if $A$ generates an analytic semigroup and if there exists
		a compact operator $F \in \mathcal{B}(X,U)$ such that 
		the semigroup generated by $A_{BF}$
		is exponentially stable, then
		Assumptions~\ref{assump:finite_multiplicities}--\ref{assump:controllability} hold,
		where the operator $A_{BF}: D(A_{BF}) \subset X \to X$ by
		\begin{equation}
		\label{eq:A_BF_def}
		A_{BF} x = (A+BF)x
		~~\text{with domain}~~
		D(A_{BF}) := \{
		x \in X : (A+BF)x \in X\}.
		\end{equation}
	}
\end{remark}

%
%

\subsubsection{Stability analysis}
For every $x \in X$,
define $x^+ := \Pi x$ and $x^- := (I-\Pi) x$.
For convenience, we use the notation 
$x^+(t)$ and $x^-(t)$ instead of 
$x(t)^+$ and $x(t)^-$, respectively.
Assume that the
feedback operator $F \in \mathcal{B}(X,U)$ satisfies
$
F^- = F|_{X^-} = 0.
$
Then the control input in \eqref{eq:input} 
is given by
\begin{equation}
\label{eq:input_unstable}
u(t) = F x(t_k) = 
F^+ x^+(t_k),\quad t_k\leq t <  t_{k+1},~k \in \mathbb{N}_0.
\end{equation}
In this subsection, we choose the time sequence $\{t_k\}_{k\in \mathbb{N}_0}$ so that 
the finite-dimensional system 
\begin{subequations}
	\label{eq:finite_state_equation}
	\begin{align}
	x^+(0) &= \Pi x^0 \in X^+ \\
	\dot x^+ (t) &= A^+x^+(t) +B^+F^+ x^+(t_k)\qquad \forall t \in(t_k, t_{k+1}),~\forall k \in \mathbb{N}_0.
	\end{align}
\end{subequations}
is exponentially stable
and its stability margin is at least $\gamma^+ > 0$.
For example, as in Theorem~\ref{thm:bounded2}, we can define
the time sequence $\{t_k\}_{k\in \mathbb{N}_0}$ by
\begin{subequations}
	\label{eq:event_trigger_present_finite}
	\begin{align}
	t_0 &:= 0,\quad
	\psi_{k+1} :=\inf \big\{t > t_{k}:\|F^+x^+(t) - F^+x^+(t_{k})\|_U > \varepsilon \|x^+(t)\|
	\big\}\\
	t_{k+1} &:= 
	\min\{ t_k + \tau_{\max},~
	\psi_{k+1}\}
	\qquad \forall k \in \mathbb{N}_0.
	\end{align}
\end{subequations}

Since $X^+$ is finite dimensional,
we obtain less conservative conditions on 
the threshold $\varepsilon$
for the closed-loop system \eqref{eq:finite_state_equation} to be exponentially stable.
In particular, we obtain the following result on the event-triggering mechanism
\eqref{eq:event_trigger_present_finite}:
\begin{proposition}
	\label{prop:finite_event_trigger_cond}
	Consider the finite-dimensional system \eqref{eq:finite_state_equation}
	and the event-triggering mechanism 
	\eqref{eq:event_trigger_present_finite}.
	Assume that the input space $U$ is finite dimensional and choose $\gamma^+ > 0$.
	If there exist positive definite matrices $P$, $Q$ and
	a positive scalar $\kappa$ such that the following linear 
	matrix inequalities are feasible:
	\begin{subequations}
		\begin{gather}
		\label{eq:LMI1}
		\begin{bmatrix}
		Q - \varepsilon^2 \kappa I & -PB^+ \\
		-(B^+)^*P& \kappa I
		\end{bmatrix}
		\succeq 0 \\
		(A^++B^+F^+)^{*} P + P(A^++B^+F^+) \preceq -(\gamma^+)^2 P - Q,
		\label{eq:LMI2}
		\end{gather}
	\end{subequations}
	then, for all $\tau_{\max}>0$,  the finite-dimensional system \eqref{eq:finite_state_equation}
	is exponentially stable and its stability margin is at least 
	$\gamma^+$.
\end{proposition}
\begin{proof}
	We can 
	prove Proposition~\ref{prop:finite_event_trigger_cond} similarly to Theorem  III.3 in \cite{Donkers2012_EV}. 
	Therefore,
	the proof is omitted.
\end{proof}

\begin{theorem}
	\label{thm:unbounded}
	Suppose that Assumptions~\ref{assump:finite_multiplicities}--\ref{assump:controllability} hold.
	Let $F^{-} = 0$ and
	assume the feedback gain $F^+$
	and the time sequence $\{t_k\}_{k\in \mathbb{N}_0}$ are chosen so that 
	the finite-dimensional system 
	\eqref{eq:finite_state_equation}
	is exponentially stable,
	its stability margin is at least $\gamma^+ > 0$, and
	there exist $\tau_{\max},\tau_{\min} >0$ such that 
	$\tau_{\min} \leq t_{k+1} - t_k \leq \tau_{\max}$ for all $x^+(0) \in X^+$ 
	and all $k \in \mathbb{N}_0$.
	Then the infinite-dimensional system \eqref{eq:plant} 
	is exponentially stable and its stability margin is at least 
	$\min\{\gamma^+, 
	-\omega(T^-)\}$.
\end{theorem}
\begin{proof}
	Setting $S_\tau^+ := \int^\tau_0 T^+(s)B^+ ds$ and 
	$S_\tau^- := \int^\tau_0 T^-(s)B^- ds$, 
	we have from \eqref{eq:unique_solution} that 
	for every $\tau \in(0, t_{k+1}-t_k]$ and every $k \in \mathbb{N}_0$,
	\begin{align}
	\label{eq:x+x-_equation}
	\begin{bmatrix}
	x^+(t_k+\tau) \\ x^-(t_k+\tau)
	\end{bmatrix} &= 
	\begin{bmatrix}
	T^+(\tau)  & 0 \\ 0 & T^-(\tau)
	\end{bmatrix}
	\begin{bmatrix}
	x^+(t_k) \\ x^-(t_k)
	\end{bmatrix} 
	+ 
	\begin{bmatrix}
	S_\tau^+ \\ S_\tau^-
	\end{bmatrix} 
	\begin{bmatrix}
	F^+ & 0
	\end{bmatrix} 	
	\begin{bmatrix}
	x^+(t_k) \\ x^-(t_k)
	\end{bmatrix}.
	\end{align}
	Choose $\gamma_0 \in (0,\min\{\gamma^+, 
	-\omega(T^-)\} )$ arbitrarily.
	The above $x^+$ is the unique solution of \eqref{eq:finite_state_equation}, and
	hence there exists $\Gamma_1 \geq 1$ such that 
	\begin{align}
	\label{eq:x_plus_bound}
	\|x^+(t)\| \leq \Gamma_1 e^{-\gamma_0 t} \|x^+(0)\|\qquad 
	\forall t \geq 0.
	\end{align}

	We will show that 
	for every $\gamma \in (0,\gamma_0)$,
	there exist $\Gamma_2,\Gamma_3 >0$ such that
	for every $\tau \in(0, t_{k+1}-t_k]$ and $k \in \mathbb{N}_0$, 
	\begin{align}
	\label{eq:x_minus_bound}
	\|x^-(t_k + \tau)\| 
	&\leq
	\Gamma_2 e^{-\gamma(t_k + \tau)} \|x^-(0)\| 
	+ \Gamma_3 e^{-\gamma(t_k+\tau)}	\|x^+(0)\|.
	\end{align}
	It follows from \eqref{eq:x+x-_equation} that
	for every $\tau \in(0, t_{k+1}-t_k]$ and every $k \in \mathbb{N}_0$,
	\begin{align*}
	x^-(t_k+\tau) &=
	T^-(t_k+\tau) x^-(0) + 
	S_{\tau}^-F^+x^+(t_k) \\
	&\quad + 
	\sum_{\ell=1}^k T^{-} (t_k+\tau-t_\ell) S^-_{t_\ell-t_{\ell-1}}F^+ x^+(t_{\ell-1}).
	\end{align*}
	Since 
	$B^- \in {\mathcal{B}(U,(X_{-1})^-)}$ and 
	since $T^-(t)$ is a strongly continuous semigroup on $(X^-)_{-1} = (X_{-1})^-$,
	it follows from Lemma~\ref{lem:Stau} that $S^-_{\tau} \in{\mathcal{B}(U,X^-)} $
	for every $\tau \geq 0$ and that there exists $L \geq 0$ such that
	\begin{equation}
	\label{eq:SmFp_bound}
	\sup_{0\leq \tau \leq \tau_{\max}}
	\left\|
	S^-_{\tau} F^+
	\right\|_{\mathcal{B}(X^+,X^-)} 
	\leq L.
	\end{equation}
	Therefore, for all 
	$\tau \in(0, t_{k+1}-t_k]$ and all $k \in \mathbb{N}_0$,
	\begin{align}
	\|x^-(t_k + \tau)\| &\leq
	\|T^-(t_k+\tau)\|_{\mathcal{B}(X^-)} \cdot \|x^-(0)\| + 
	L  \|x^+(t_k)\|  \notag \\
	&\qquad+
	\sum_{\ell=1}^k
	\|T^-(t_k+\tau - t_\ell)\|_{\mathcal{B}(X^-)}  \cdot 
	L  \|x^+(t_{\ell-1})\|.
	\label{eq:x_tk_tau_norm}
	\end{align}
	By the exponential stability of $T^-(t)$,
	there exists $M_1 \geq 1$ such that 
	\begin{equation}
	\label{eq:T_minum_bound}
	\|T^-(t)\|_{\mathcal{B}(X^-)} \leq M_1 e^{-\gamma_0 t}
	\qquad \forall t \geq 0.
	\end{equation}
	From \eqref{eq:x_plus_bound}, \eqref{eq:x_tk_tau_norm}, and \eqref{eq:T_minum_bound},
	for all $\tau \in(0, t_{k+1}-t_k]$ and all $k \in \mathbb{N}_0$,
	\begin{align*}
	\|x^-(t_k + \tau)\| &\leq
	M_1 e^{-\gamma_0 (t_k + \tau)} \|x^-(0)\| 
	+
	L\Gamma_1 e^{-\gamma_0 t_k} \|x^+(0)\| \\
	&\qquad +
	\sum_{\ell=1}^k
	M_1 e^{-\gamma_0 (t_k + \tau - t_\ell)} \cdot 
	L\Gamma_1 e^{-\gamma_0 t_{\ell-1}}\|x^+(0)\|  \\
	&\leq
	M_1 e^{-\gamma_0 (t_k + \tau)} \|x^-(0)\| 
	+
	(k+1) L\Gamma_1M_1 e^{-\gamma_0 (t_k+\tau - \tau_{\max})}\|x^+(0)\|.
	\end{align*}
	Note that
	$k\in \mathbb{N}_0$ is  
	the number of the control updates during $[0,t_k+\tau)$ 
	for all $\tau \in (0,t_{k+1}-t_k]$ and hence satisfies
	\begin{equation}
	\label{eq:event_number_cond}
	k \leq \frac{t_k+\tau}{\tau_{\min}} \qquad
	\forall \tau \in (0,t_{k+1} - t_k].
	\end{equation}
	We also have that, for every $\gamma \in (0, \gamma_0)$, there exists
	$M_2 \geq 1$ such that 
	\[
	\left(
	\frac{t}{\tau_{\min}} + 1
	\right) e^{-\gamma_0 t} \leq M_2 e^{-\gamma t}\qquad \forall t \geq 0.
	\]
	Then for every $\tau \in(0, t_{k+1}-t_k]$ and $k \in \mathbb{N}_0$, 
	\eqref{eq:x_minus_bound} holds with
	$\Gamma_2 := M_1$ and 
	$\Gamma_3 := L\Gamma_1 M_1M_2 e^{\gamma_0\tau_{\max}}$.

	Since
	\[
	\|x^+(0)\| \leq \|\Pi\|_{\mathcal{B}(X)}  \cdot \|x^0\|,\quad
	\|x^-(0)\| \leq \big(1+\|\Pi\|_{\mathcal{B}(X)} \big) \cdot \|x^0\|\qquad \forall x^0 \in X,
	\]
	it follows from \eqref{eq:x_plus_bound} and 
	\eqref{eq:x_minus_bound}
	that 
	\begin{align*}
	\|x^+(t)\| \leq M^+ e^{-\gamma t} \|x^0\|,\quad
	\|x^-(t)\| \leq M^- e^{-\gamma t} \|x^0\|
	\qquad \forall x^0 \in X,~\forall t \geq 0,
	\end{align*}
	where
	$M^+ := \Gamma_1 \|\Pi\|_{\mathcal{B}(X)}$ and 
	$M^- := \Gamma_2 + (\Gamma_2 +\Gamma_3) \|\Pi\|_{\mathcal{B}(X)}.$
	Thus we obtain
	\begin{align*}
	\|x(t)\| &=
	\|x^+(t) + x^-(t)\| 
	\leq 
	\|x^+(t)\| + \|x^-(t)\| 
	\leq 
	(M^+ + M^-) e^{-\gamma t} \|x^0\|
	\end{align*}
	for all $x^0 \in X$ and all $t \geq 0$.
	Thus, 
	the infinite-dimensional system \eqref{eq:plant}
	is exponentially stable.
	Since the constants
	$\gamma_0 \in (0,\min\{\gamma^+, 
	-\omega(T^-)\} )$ and $\gamma \in (0,\gamma_0)$ were arbitrary,
	the stability margin
	is at least
	$\min\{\gamma^+, 
	-\omega(T^-)\}$.
\end{proof}

\subsection{Periodic event-triggered control}
In Theorem~\ref{thm:unbounded}, the feedback operator $F \in \mathcal{B}(X,U)$
has a specific structure $F^- = F|_{X^-} = 0$. In contrast, we here assume 
that $A$ generates
an analytic semigroup, and use a periodic event-triggering condition 
proposed in
\cite{Heemels2013, Heemels2013_Automatica}. Then we see that
for every compact feedback operator $F \in \mathcal{B}(U,X)$ for which
the semigroup generated by $A_{BF}$ in \eqref{eq:A_BF_def}
is exponentially stable,
there exists a periodic event-triggering condition
such that the closed-loop system \eqref{eq:plant} is exponentially stable.

Fixing $h>0$, $\varepsilon >0$, and $\ell_{\max} \in \mathbb{N}$, 
we define the time sequence $\{t_k\}_{k\in \mathbb{N}_0}$ by
\begin{subequations}
	\label{eq:PET}
	\begin{align}
	t_0 &:= 0,\quad 
	\psi_k := \min \big\{ \ell h > t_k:
	\|Fx(\ell h) - Fx(t_k)\|_U > \varepsilon \|x(t_k)\|,~\ell \in \mathbb{N}
	\big\} \\
	t_{k+1} &:= 
	\min\{ t_k + \ell_{\max}h,~\psi_k
	\}
	\qquad \forall k \in \mathbb{N}_0,
	\end{align}
\end{subequations}
which is a class of periodic event-triggering mechanisms
\cite{Heemels2013, Heemels2013_Automatica}.
Whereas the event-triggering mechanisms \eqref{eq:time_seq} 
and \eqref{eq:time_seq_present} require monitoring of
the conditions continuously, 
the periodic event-triggering mechanism \eqref{eq:PET}
verifies the conditions only periodically.


The result of this section is based on the following theorem on the 
exponential stability of periodic sampled-data systems:
\begin{theorem}[Theorem~4.8 in \cite{Logemann2003}]
	\label{thm:sampled_data}
	Assume that $A$ generates an analytic semigroup $T(t)$ on X,
	$B \in \mathcal{B}(U,X_{-1})$, and $F \in \mathcal{B}(X,U)$ is compact. If
	the semigroup generated by $A_{BF}$ in \eqref{eq:A_BF_def}
	is exponentially stable, then
	there exists $h^* >0$ such that for every $h \in (0,h^*)$,
	the periodic sampled-data system \eqref{eq:plant}  with $t_{k+1} - t_k = h$, 
	$k \in \mathbb{N}_0$, is exponentially stable.
\end{theorem}

Theorem~\ref{thm:PET_ex_stability} below shows the existence of 
periodic event-triggering mechanisms
achieving the exponential stability of the closed-loop system \eqref{eq:plant}.
\begin{theorem}
	\label{thm:PET_ex_stability}
	Assume that $A$ generates an analytic semigroup $T(t)$ on X,
	$B \in \mathcal{B}(U,X_{-1})$, and $F \in \mathcal{B}(X,U)$ is compact.
	Assume further that
	the semigroup generated by $A_{BF}$ in \eqref{eq:A_BF_def}
	is exponentially stable.
	Choose $h > 0$ so that  the periodic sampled-data system \eqref{eq:plant}  with $t_{k+1} - t_k = h$, 
	$k \in \mathbb{N}_0$, is exponentially stable.
	Then there exists $\varepsilon^* >0$ such that
	the system \eqref{eq:plant} with the
	periodic event-triggering mechanism \eqref{eq:PET}
	is exponentially stable for every $\varepsilon \in (0,\varepsilon^*)$ and every $\ell_{\max} \in \mathbb{N}$.
\end{theorem}

\begin{proof}
	Theorem~\ref{thm:sampled_data} guarantees 
	the existence of $h>0$ such that the periodic sampled-data system \eqref{eq:plant}  with $t_{k+1} - t_k = h$, 
	$k \in \mathbb{N}_0$, is exponentially stable.
	By Lemma~2.3  in \cite{Logemann2003}, this exponential stability is achieved if and only if
	the operator $\Delta(h) := T(h) + S_h F \in \mathcal{B}(X)$ is power stable,
	that is, there exist $M \geq 1$ and $\delta \in (0,1)$ such that
	\[
	\|\Delta(h)^k \|_{\mathcal{B}(X)}  \leq M \delta^k \qquad \forall k \in \mathbb{N}_0.
	\]
	
	For the time sequence $\{t_k\}_{k\in \mathbb{N}_0}$ defined by \eqref{eq:PET},
	let $\ell_k  \in \mathbb{N}_0$ satisfy $t_k = \ell_k h$.
	We have from \eqref{eq:unique_solution} that
	for all $p \in \{0,\dots, \ell_{k+1} - \ell_k\}$ and 
	all $k \in \mathbb{N}_0$,
	\begin{align}
	\label{eq:x_TS_dis}
	x\big((\ell_k + p + 1)h\big) &= T(h) x\big((\ell_k + p)h\big) + S_hF x(\ell_kh).
	\end{align}
	Since  $S_h \in \mathcal{B}(U,X)$ by Lemma~\ref{lem:Stau},
	we can regard the following proof of Theorem~\ref{thm:PET_ex_stability} 
	as the discrete-time counterpart
	of Theorem~\ref{thm:bounded1}.
	
	For $p \in \{0,\dots, \ell_{k+1} - \ell_k - 1\}$ and
	$k \in \mathbb{N}_0$,
	define the error $e$ by
	\begin{equation}
	\label{eq:IIE_def}
	e\big((\ell_k + p)h\big) := Fx\big((\ell_k + p)h\big) -  F x(\ell_kh).
	\end{equation}
	By \eqref{eq:x_TS_dis},
	\begin{align*}
	x\big((\ell_k + p + 1)h\big)	
	&=
	\Delta(h) x\big((\ell_k + p)h\big) - S_he\big((\ell_k + p)h\big).
	\end{align*}
	for every $p \in \{0,\dots, \ell_{k+1} - \ell_k - 1\}$ and every
	$k \in \mathbb{N}_0$.
	Applying induction the equation above, we obtain
	\begin{equation}
	\label{eq:x_ell_k_1}
	x(\ell_{k+1}h)=
	\Delta(h)^{\ell_{k+1} - \ell_k} x(\ell_k h) 
	- \sum_{p=0}^{\ell_{k+1} - \ell_k-1}
	\Delta(h)^{\ell_{k+1} - \ell_k - p -1} S_h e\big((\ell_k + p)h\big)
	\end{equation}
	for 
	all $k \in \mathbb{N}_0$.

	We introduce a new norm $|\cdot|_{\rm d}$ on $X$ defined by
	\[
	|x|_{\rm d} := \sup_{\ell \in \mathbb{N}_0} \|\delta^{-\ell} \Delta(h)^\ell x\|.
	\]
	As in \eqref{eq:new_norm_prop} for the continuous-time counterpart,
	this norm has the following properties:
	\begin{align}
	\label{eq:norm_d_property}
	\|x\| \leq |x|_{\rm d} \leq M\|x\|,\quad 
	|\Delta(h)^k x|_{\rm d} \leq \delta^k |x|_{\rm d} \qquad 
	\forall x \in X,~\forall k\in \mathbb{N}_0.
	\end{align}
	Under the  event-triggering mechanism \eqref{eq:PET}, 
	the error $e$ given in \eqref{eq:IIE_def}
	satisfies
	\begin{equation}
	\label{eq:IIE_bound}
	\big\|e\big((\ell_k + p)h\big)\big\| \leq \varepsilon \|x(\ell_kh)\|\qquad 
	\forall p \in \{0,\dots, \ell_{k+1} - \ell_k -1\},~
	\forall k \in \mathbb{N}_0.
	\end{equation}
	Combining \eqref{eq:x_ell_k_1}--\eqref{eq:IIE_bound}, we obtain
	\begin{align*}
	|x(\ell_{k+1}h)|_{\rm d} &\leq 
	\delta^{\ell_{k+1} - \ell_k} |x(\ell_k h)|_{\rm d} +  \varepsilon M \|S_h\|_{\mathcal{B}(U,X)} 
	\sum_{p=0}^{\ell_{k+1} - \ell_k - 1}
	\delta^{\ell_{k+1} - \ell_k - p - 1}|x(\ell_k h)|_{\rm d} \\
	&= 
	\left(
	\delta^{\ell_{k+1} - \ell_k} (1-\varepsilon_0) + \varepsilon_0
	\right) \cdot |x(\ell_k h)|_{\rm d}\qquad \forall k \in \mathbb{N}_0,
	\end{align*}
	where 
	\[
	\varepsilon_0 :=  \varepsilon \frac{M\|S_h\|_{\mathcal{B}(U,X)} }{1-\delta}.
	\]
	Choose the threshold $\varepsilon >0$
	so that 
	$\delta(1-\varepsilon_0) + \varepsilon_0 < 1$, i.e., 
	\[
	\varepsilon < \frac{1-\delta}{M\|S_h\|_{\mathcal{B}(U,X)} }.
	\]
	Define the function $f(\ell)$ by
	\[
	f(\ell) := \frac{-\log (\delta^\ell (1-\varepsilon_0) + \varepsilon_0)}{\ell h}.
	\]
	Then $f(\ell)$ is positive and monotonically decreasing in $\mathbb{N}$.
	Applying \eqref{eq:norm_d_property}, we obtain
	\begin{align}
	\notag
	\|x(\ell_{k+1}h)\| 
	\leq
	|x(\ell_{k+1}h)|_{\rm d}\
	&\leq e^{-\gamma (\ell_{k+1} - \ell_k)h} |x(\ell_kh)|_{\rm d} \\
	\label{eq:xl_k1h_bound}
	&\leq	
	e^{-\gamma \ell_{k+1}h} |x^0|_{\rm d} \leq M e^{-\gamma \ell_{k+1}h} \|x^0\|
	\qquad \forall k \in \mathbb{N}_0,
	\end{align}
	where $\gamma := f(\ell_{\max}) >0$.
	
	From Lemma~\ref{lem:Stau}, there exists $L \geq 1$ such that $\|\Delta(\tau)\|_{\mathcal{B}(X)} \leq L$
	for every $\tau \in [0,\ell_{\max}h]$.
	Therefore, we see from \eqref{eq:unique_solution} and \eqref{eq:xl_k1h_bound}
	that 
	\begin{align*}
	\|x(\ell_kh + \tau)\| 
	&\leq	
	L \|x(\ell_kh)\| \\
	&\leq 
	\left(
	e^{\gamma \ell_{\max} h} LM 
	\right) \cdot
	e^{-\gamma (\ell_k h +\tau)} \|x^0\|\qquad \forall \tau \in 
	(0,(\ell_{k+1} - \ell_k)h],~\forall k \in \mathbb{N}_0.
	\end{align*}
	Thus, the  system \eqref{eq:plant} with the
	periodic event-triggering mechanism \eqref{eq:PET}
	is exponentially stable and its stability margin is at least $\gamma > 0$.
\end{proof}
\begin{remark}
	{\em
		As easily seen from the proof of Theorem~\ref{thm:PET_ex_stability},
		one can obtain
		a similar result for the following event-triggering
		mechanism that uses the error  of the state:
		\begin{subequations}
			\begin{align*}
			t_0 &:= 0,\quad 
			\psi_k := 
			\min\big\{\ell h > t_k:
			\|x(\ell h) - x(t_k)\| > \varepsilon \|x(t_k)\|,~\ell \in \mathbb{N} 
			\big\} \\
			t_{k+1} &:= 
			\min \{ t_k + \ell_{\max}h,~ \psi_k\}
			\qquad \forall k \in \mathbb{N}_0.
			\end{align*}
		\end{subequations}
	}
\end{remark}

\begin{remark}
	{\em
		In Theorem~\ref{thm:PET_ex_stability}, we assume that
		$A$ generates an analytic semigroup and that $F$ is compact.
		These assumptions are used for the existence of 
		sampling periods with respect to which
		the periodic sampled-data system is exponentially stable. We 
		can replace them with different assumptions such as
		those of Corollary 2.3 in \cite{Rebarber2006}.
	}
\end{remark}

\section{Numerical examples}
In this section, we provide numerical examples for both the case 
$B \in \mathcal{B}(U,X)$ and $B \not\in \mathcal{B}(U,X)$.
We consider a cascade ODE-PDE system for bounded $B$ and
an Euler-Bernoulli beam for unbounded $B$.
\subsection{Bounded control}
We illustrate the event-triggering mechanism in Theorem~\ref{thm:bounded1} with
a heat PDE in cascade with an ODE.
Let $b = 
\begin{bmatrix}
b_1 & \cdots & b_n
\end{bmatrix}\in L^2([0,1],\mathbb{R})^{1 \times n}$,
$G \in \mathbb{R}^{n \times n}$, and $H \in \mathbb{R}^{n \times m}$.
For the space variable $\xi \in [0,1]$ and the time variable $t \geq 0$,
we consider the following ODE-PDE system:
\begin{subequations}
	\label{eq:ODE_PDE}
	\begin{align}
	&\frac{\partial z_1}{\partial t}(\xi,t) = 
	\frac{\partial^2 z_1}{\partial \xi^2}(\xi,t) + b(\xi)z_2(t),\quad \xi \in [0,1],~t \geq 0 \\
	& \frac{\partial z_1}{\partial \xi}(0,t) = 0,\quad \frac{\partial z_1}{\partial \xi}(1,t) = 0,\quad t\geq 0;\quad
	z_1(\xi,0) = z_1^0(\xi),\quad \xi \in [0,1] \\
	& \dot z_2(t) = Gz_2(t) + Hu(t),\quad t\geq 0;\qquad
	z_2(0) = z_2^0,
	\end{align}
\end{subequations}
where 
$z_1(\xi,t)$ is the temperature at position $\xi \in [0,1]$ and time $t\geq0$,
$z_2(t)$ is the state of the ODE, and
$u(t)$ is the input.

\subsubsection{Exponential stability of continuous-time closed-loop system}
We can reformulate the cascaded ODE-PDE \eqref{eq:ODE_PDE} as 
an abstract evolution equation \eqref{eq:state_equation} in the following way.
We write $L^2(0,1)$ in place of $L^2([0,1],\mathbb{C})$.
Define the state space $X$ and the input space $U$ by
$X :=   L^2(0,1) \times \mathbb{C}^n$ and $U := \mathbb{C}^m$.
Let $z_1^0 \in L^2(0,1) $ and
$z_2^0 \in \mathbb{C}^n$.
The state space $X$ is a Hilbert space with the inner product
\[
\left\langle
\begin{bmatrix}
x_1 \\ x_2
\end{bmatrix},~
\begin{bmatrix}
y_1 \\ y_2
\end{bmatrix}
\right\rangle :=
\big\langle x_1,~
y_1
\big\rangle_{L^2} + 
\left\langle
x_2,~
y_2
\right\rangle_{\mathbb{C}^n}.
\]
Set
\[
x(t) :=  
\begin{bmatrix}
x_1(t) \\ x_2(t)
\end{bmatrix}
\text{~~with~~$x_1(t) := z_1(\cdot, t)$ and $x_2(t) := z_2(t)$};
\qquad x^0 :=
\begin{bmatrix}
z_1^0 \\ z_2^0
\end{bmatrix} \in X.
\]
Define $\phi_0(\xi) := 1$ and $\phi_n (\xi) := \sqrt{2} \cos(n \pi\xi )$ for 
$n \in \mathbb{N}$, which
form an orthonormal basis for $L^2(0,1)$.
Define 
$A_1 : D(A_1) \subset L^2(0,1) \to L^2(0,1)$ by
\begin{equation}
\label{eq:A1_expression}
A_1x_1 :=-\sum_{n = 0}^\infty 
n^2\pi^2
\langle
x_1,\phi_n
\rangle_{L^2} \phi_{n} 
\end{equation}
with domain
\[
D(A_1) := \left\{
x_1 \in L^2(0,1):
\sum_{n = 0}^\infty  n^4\pi^4 
|\langle
x_1,\phi_n
\rangle_{L^2}|^2 < \infty
\right\}
\]
and $B_1 : \mathbb{C}^n \to L^2(0,1)$  by
\[
B_1x_2 :=  bx_2\qquad \forall x_2 \in \mathbb{C}^n.
\]
If we set
\[
A := 
\begin{bmatrix}
A_1 & B_1 \\
0 & G
\end{bmatrix}\text{~~with~~$D(A) := D(A_1) \times \mathbb{C}^n$},\qquad 
B := 
\begin{bmatrix}
0 \\
H
\end{bmatrix},
\]
then we can rewrite the cascaded ODE-PDE \eqref{eq:ODE_PDE} as
an abstract evolution equation in the form of \eqref{eq:state_equation};
see Example~2.3.7 on p.~45 in \cite{Curtain1995}
for the expansion \eqref{eq:A1_expression} of $A_1$.

Consider the situation where the state $x(t)$ is observed at all $t\geq0$. 
In the case of continuous-time control,
we generate the input $u$ as follows:
\begin{equation}
\label{eq:ODE_PDE_control}
u(t) = F x(t),\quad t \geq 0,
\end{equation}
where $
F := 
\begin{bmatrix}
F_1 & F_2
\end{bmatrix}
$
with $F_1 \in \mathcal{B}(L^2(0,1),\mathbb{C}^m)$ and $F_2 \in \mathbb{C}^{m \times n}$.
In this case, the dynamics of the closed-loop system is given by
\begin{equation}
\label{eq:ODE_PDE_closed}
\dot x (t) = 
(A+BF) x(t),\quad t \geq 0;\qquad x(0) = x^0 \in X.
\end{equation}

The following proposition provides a sufficient condition for
the strongly continuous semigroup $T_{BF}(t)$ generated by $A+BF$ to be
exponentially stable.

\begin{proposition}
	\label{prof:gain_cond}
	Consider the evolution equation \eqref{eq:ODE_PDE_closed} obtained from 
	the ODE-PDE system \eqref{eq:ODE_PDE} and the state-feedback 
	controller \eqref{eq:ODE_PDE_control} 
	as above.
	Suppose that $F_1 x_1 = f\langle x_1,\phi_0 \rangle_{L^2}$ for some $f \in \mathbb{C}^m$.
	The strongly continuous semigroup $T_{BF}(t)$ generated by 
	$A+BF$ is exponentially stable if the rational transfer functions
	\[
	{\bf G}_1(\lambda)
	:= \big(
	\lambda  - B^+ (\lambda I-G-HF_2)^{-1} Hf
	\big)^{-1},\quad {\bf G}_2(\lambda)
	:= (\lambda I-G-HF_2)^{-1} 
	\]
	have poles only in $\mathbb{C}_-$,
	where $B^+ := 
	\begin{bmatrix} 
	\langle b_1,\phi_0 \rangle_{L^2} &\cdots&
	\langle b_n,\phi_0 \rangle_{L^2}
	\end{bmatrix}
	$.
\end{proposition}

\begin{proof}
	From the strong continuity of $T_{BF}(t)$,
	it follows that $x(t) := T_{BF}(t)x^0$ is continuous for every $t\geq 0$.
	Moreover, there exist $M \geq 1$ and $\omega \in \mathbb{R}$ such that 
	$\|x(t)\| \leq M e^{\omega t}$ for every $t\geq 0$.
	
	Let $T_1(t)$ be the strongly continuous semigroup 
	generated by $A_1$. Then
	\[
	T_1(t) x_1 = \sum_{n = 0}^\infty 
	e^{-n^2\pi^2 t}
	\langle
	x_1,\phi_n
	\rangle_{L^2} \phi_{n}\qquad \forall x_1 \in L^2(0,1),~\forall t\geq 0.
	\]
	Since
	\begin{equation}
	\label{eq:x1_ODE_PDE}
	x_1(t) = T_1(t) z_1^0 + \int^t_0 T_1(t-s)B_1x_2(s)ds\qquad \forall t\geq 0,
	\end{equation}
	it follows that 
	\begin{equation}
	\label{eq:zeta_diff}
	x^+_1(t) = \langle z_1^0 ,\phi_0 \rangle_{L^2} + 
	\int^t_0 B^+x_2(s) ds\qquad \forall t\geq 0,
	\end{equation}
	where $x^+_1(t):= \langle x_1(t), \phi_0 \rangle_{L^2}$ and 
	$B^+ := 
	\begin{bmatrix} 
	\langle b_1,\phi_0 \rangle_{L^2} &\cdots&
	\langle b_n,\phi_0 \rangle_{L^2}
	\end{bmatrix}$.

	On the other hand,
	\begin{equation}
	\label{eq:x_solution}
	x_2(t) = e^{(G+HF_2)t} z_2^0 + \int^t_0 e^{(G+HF_2)(t-s)} Hf x_1^+(s)ds\qquad \forall t\geq 0.
	\end{equation}
	Substituting this into \eqref{eq:zeta_diff}, we obtain
	\[
	x_1^+(t) = \langle z_1^0 ,\phi_0 \rangle_{L^2}  + 
	\int^t_0 \left(
	B^+ e^{(G+HF_2)s} z_2^0 +\int^s_0  B^+e^{(G+HF_2)(s-p)} Hf x_1^+(p)dp\right) ds
	\]
	for every $t\geq 0$.
	This integral equation can be solved by the Laplace transform.
	Denote by $\mathcal{L}[x_1^+]$ the Laplace transform of $x_1^+$.
	For every sufficiently large $\lambda >0$,
	we obtain
	\begin{align*}
	\lambda \mathcal{L}[x_1^+](\lambda) &= \langle z_1^0 ,\phi_0 \rangle_{L^2}  + 
	B^+ (\lambda I-G-HF_2)^{-1} z_2^0 \\
	&\qquad\qquad  + 
	B^+ (\lambda I-G-HF_2)^{-1} Hf \mathcal{L}[x_1^+](\lambda),
	\end{align*}
	and hence
	\begin{equation}
	\label{eq:x1p_laplace}
	\mathcal{L}[x_1^+](\lambda) = {\bf G}_1(\lambda)
	\big(
	\langle z_1^0 ,\phi_0 \rangle_{L^2} +
	B^+ {\bf G}_2(\lambda) z_2^0
	\big),
	\end{equation}
	where
	${\bf G}_1(\lambda) := \big(
	\lambda  - B^+ (\lambda I-G-HF_2)^{-1} Hf
	\big)^{-1}$ and ${\bf G}_2(\lambda) := (\lambda I-G-HF_2)^{-1} $.
	
	Suppose that ${\bf G}_1(\lambda)$ and ${\bf G}_2(\lambda)$ 
	have poles only in $\mathbb{C}_-$.
	By \eqref{eq:x1p_laplace},
	there exist $M_1\geq 1$ and $\omega_1 >0$ such that 
	$|x_1^+(t)| \leq M_1 e^{-\omega_1 t} \|x^0\|$ for every $t\geq 0$.
	This together with \eqref{eq:x_solution} shows that 
	there exist $M_2\geq 1$ and $0< \omega_2 \leq \omega_1$ such that 
	$\|x_2(t)\|_{\mathbb{C}^n} \leq M_2e^{-\omega_2 t} \|x^0\|$
	for every $t\geq 0$.
	
	Define the projection operator $\Pi$ on $L^2(0,1)$ by
	$\Pi x_1 = \langle x_1, \phi_0 \rangle \phi_0$ for  $x_1 \in L^2(0,1)$.
	Set $x_1^- := (I - \Pi)x_1$, $A_1^- := A_1|_{(I - \Pi)L^2(0,1)}$, and 
	$B_1^- := (I-\Pi) B_1$.
	Lemma~2.5.7 on p.~71 in \cite{Curtain1995} 
	shows that 
	$T_1^-(t) :=  T_1(t)|_{(I - \Pi)L^2(0,1)}$ is the strongly continuous semigroup 
	generated by $A_1^-$ and is exponentially stable.
	Using \eqref{eq:x1_ODE_PDE}, we obtain
	\[
	x_1^-(t) = T_1^-(t) (I - \Pi)z^0_1 + \int^t_0 T_1^-(t-s) B_1^- x_2(s)ds\qquad \forall t\geq 0.
	\]
	Therefore,
	there exist $M_3\geq 1$ and $0< \omega_3 \leq \omega_2$ such that 
	$\|x_1^-(t)\| \leq M_3 e^{-\omega_3 t} \|x^0\|$
	for every $t\geq 0$.
	Thus, $T_{BF}(t)$ is exponentially stable.
\end{proof}

\subsubsection{Numerical simulation}
Let $n=m=1$, $F_1 x_1 = f\langle x_1,\phi_0 \rangle_{L^2}$ for some $f \in \mathbb{R}$, and
$F_2 \in \mathbb{R}$.
Since $\mathbf{G}_1(\lambda)$ in Proposition~\ref{prof:gain_cond}
is given by
\[
\mathbf{G}_1(\lambda)
=
\frac{\lambda - G-HF_2}{\lambda^2 - (G+HF_2)\lambda - B^+Hf},
\]
it follows that
$T_{BF}(t)$ is exponentially stable if  $G+HF_2 < 0$ and  $B^+Hf < 0$.

Set
\[
b = b_1 = 5 \mathds{1}_{[0.4,0.6]},~~G = 0.5,~~H=1,~~f = -1,~~F_2 = -2.5,
\]
where $\mathds{1}_{[0.4,0.6]}$ is 
the indicator function of the interval $[0.4,0.6]$.
We obtain $B^+ = \langle b,\phi_0 \rangle_{L^2} =  1$, and hence $T_{BF}(t)$ is exponentially stable. Furthermore,
as seen in the proof of Proposition~\ref{prof:gain_cond},
for every $\omega \in (0,1]$,
there exists $M\geq 1$ such that 
\begin{equation}
\label{eq:M_bound}
\|T_{BF}(t)\|_{\mathcal{B}(X)}  \leq M e^{-\omega t} \qquad \forall t\geq 0.
\end{equation}
We here find a constant $M = M(\omega)$ satisfying \eqref{eq:M_bound} 
by approximation as follows.
For $N \in \mathbb{N}$, define the projection operator $\Pi_N$ on $L^2(0,1)$ by
\[
\Pi_N x_1 := \sum_{n=0}^N \langle x_1,\phi_n \rangle_{L^2} \phi_n \qquad \forall x_1 \in L^2(0,1)
\]
and the operators $A_1(N),B_1(N)$, and $F_1(N)$ on finite-dimensional spaces by
\[
A_1(N) := A_1|_{\Pi_N L^2(0,1)},~~B_1(N) := \Pi_N B_1,~~
F_1(N) := {F_1}|_{\Pi_N L^2(0,1)}.
\]
Set 
\[
A(N) := 
\begin{bmatrix}
A_1(N) & B_1(N) \\
0_{1\times (N+1)} & G
\end{bmatrix},~
B(N) :=
\begin{bmatrix}
0_{(N+1)\times 1}\\
H
\end{bmatrix},~
F(N) := 
\begin{bmatrix}
F_1(N) & F_2 
\end{bmatrix}.
\]
By the same argument in the proof of Proposition~\ref{prof:gain_cond},
we have that for every $N \in \mathbb{N}$ and every $\omega \in (0,1]$, there exists $M(N,\omega) \geq 1$ such that 
\begin{equation}
\label{eq:approximation_M_bound}
\left\|e^{(A(N)+B(N)F(N))t} \right\|_{\mathbb{C}^{(N+2) \times (N+2)}}  \leq M(N,\omega) e^{-\omega t}\qquad
\forall t \geq 0.
\end{equation}
For $N \in \mathbb{N}$ and  $\omega \in (0,1]$, set 
\begin{equation}
\label{eq:Mmin_def}
M_{\min}(N,\omega) := 
\sup_{t \geq 0} 
e^{\omega t}\left\|e^{(A(N)+B(N)F(N))t} \right \|_{\mathbb{C}^{(N+2) \times (N+2)}},
\end{equation}
which can be computed numerically.
We choose a constant $M = M(\omega)$ in \eqref{eq:M_bound} so that
\[
M \geq  \limsup_{N\to \infty} M_{\min}(N,\omega).
\]

Set $\omega = 0.5$. 
As shown in Fig.~\ref{fig:M_plot}, we numerically see that 
$M_{\min}(N,0.5)$ converges to the value less than
$1.571$ as $N \to \infty$.
Therefore we here assume that $M = 1.571$ satisfies \eqref{eq:M_bound} with $\omega = 0.5$.
By Theorem~\ref{thm:bounded1},
if the threshold $\varepsilon$ of  the event-triggering mechanism  \eqref{eq:time_seq_upper_bound} 
satisfies  $\varepsilon \leq 0.31$, then
the system \eqref{eq:plant} with  this event-triggering mechanism
is exponentially stable for every $\tau_{\max} > 0$.

For the time responses, the initial states $z_1^0$ and $z_2^0$ are given by
$
z_1^0(\xi) \equiv 1$ and $z_2^0 = -1$, respectively.
In the simulations, we approximate  $L^2(0,1)$ by 
the linear span of $\{\phi_n: n \in \mathbb{N}_0, n\leq 20\}$.
Fig.~\ref{fig:response_ODE_PDE} depicts the state norm $\|x(t)\|$ and the input $u(t)$ 
by the event-triggering mechanism
\eqref{eq:time_seq_upper_bound} with  $\varepsilon = 0.3$ and $\tau_{\max} = 1$.
As a comparison, we also show in Fig.~\ref{fig:response_ODE_PDE} 
the case under periodic sampled-data control with $t_{k+1} - t_k \equiv 0.4$.
We see from Fig.~\ref{fig:state_ODE_PDE} that the event-triggering mechanism
achieves faster convergence of $\|x(t)\|$ than the conventional periodic mechanism.
Define $T_{\rm s} := \sup\{
t\geq 0: \|x(t)\| > 0.05 \|x^0\|
\}$. 
A small $T_{\rm s}$ means  the fast convergence of the state.
We obtain $T_{\rm s} = 3.838$ under the event-triggering mechanism and
$T_{\rm s} = 4.061$ under the conventional periodic mechanism.
On the other hand, the numbers of the input updates on $(0,T_{\rm s})$ are 
$10$ under both the mechanisms.
This implies that 
the event-triggering mechanism achieves faster state convergence, by
efficiently updating the control input.
In fact,
we observe from Fig~\ref{fig:input_ODE_PDE} that 
the event-triggering mechanism updates the control input more frequently 
on the interval $[0,1.5]$ when the (relative) change of $Fx(t)$ is large, but
less frequently on the interval $[2,5]$ when the change of $Fx(t)$ is small.


\begin{figure}[tb]
	\centering
	\includegraphics[width = 6.3cm]{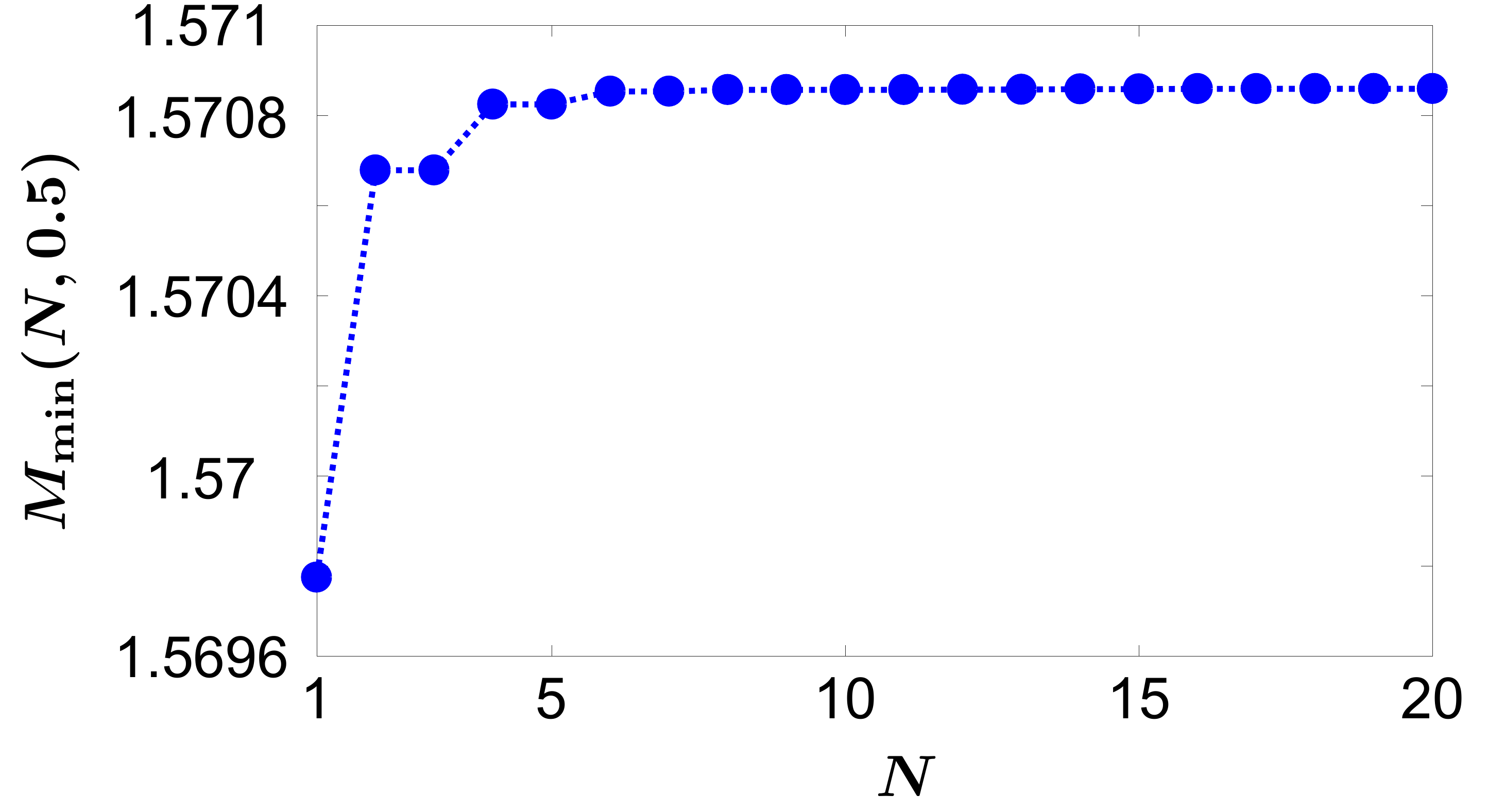}
	\caption{Constant $M_{\min}(N,0.5)$ in \eqref{eq:Mmin_def}.}
	\label{fig:M_plot}
\end{figure}

\begin{figure}
	\centering
	\subcaptionbox{State norm $\|x(t)\|$.
		\label{fig:state_ODE_PDE}}
	[.49\linewidth]
	{\includegraphics[width = 6.3cm,clip]{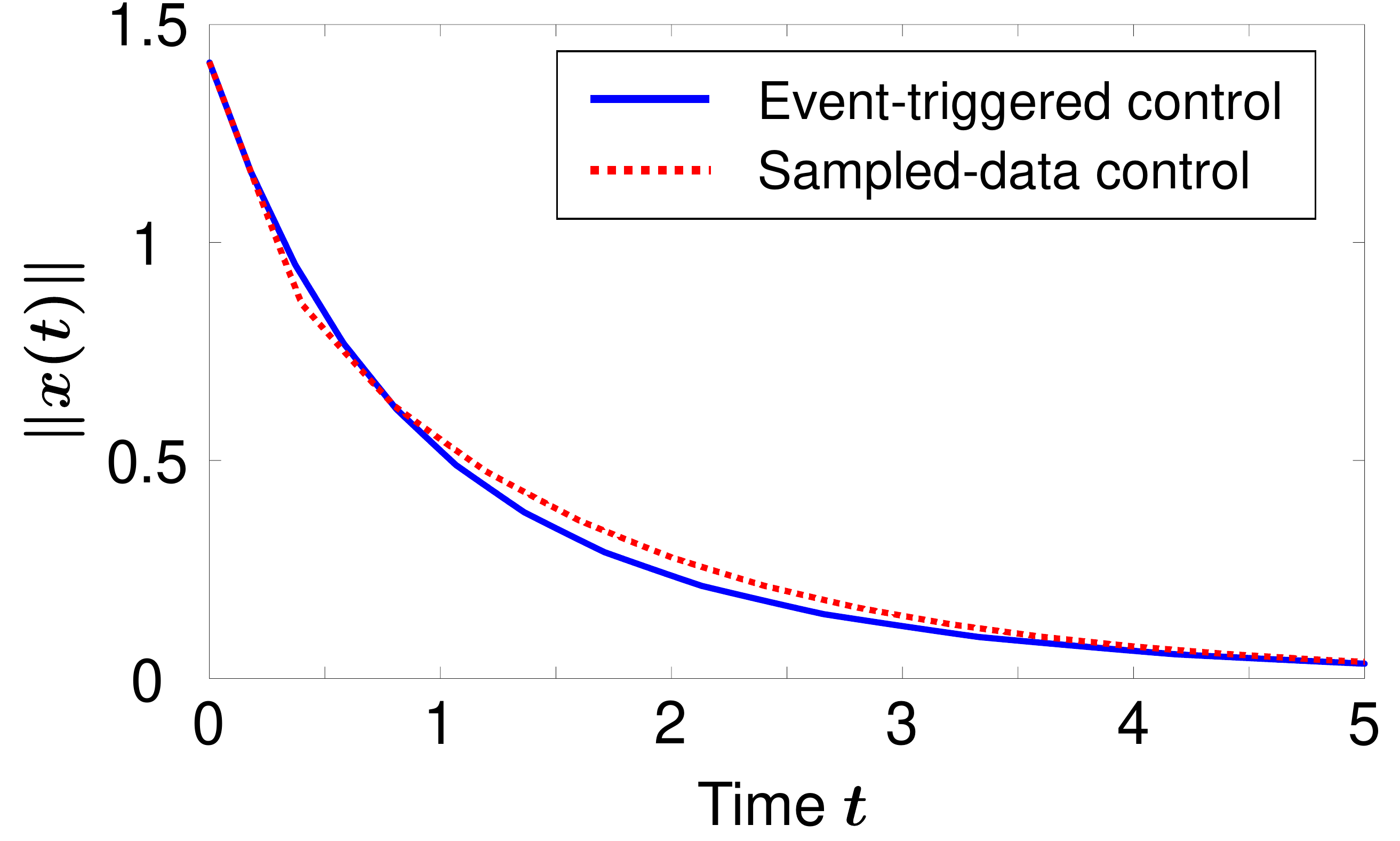}} 
	\subcaptionbox{Input $u(t)$.
		\label{fig:input_ODE_PDE}}
	[.5\linewidth]
	{\includegraphics[width = 6.3cm,clip]{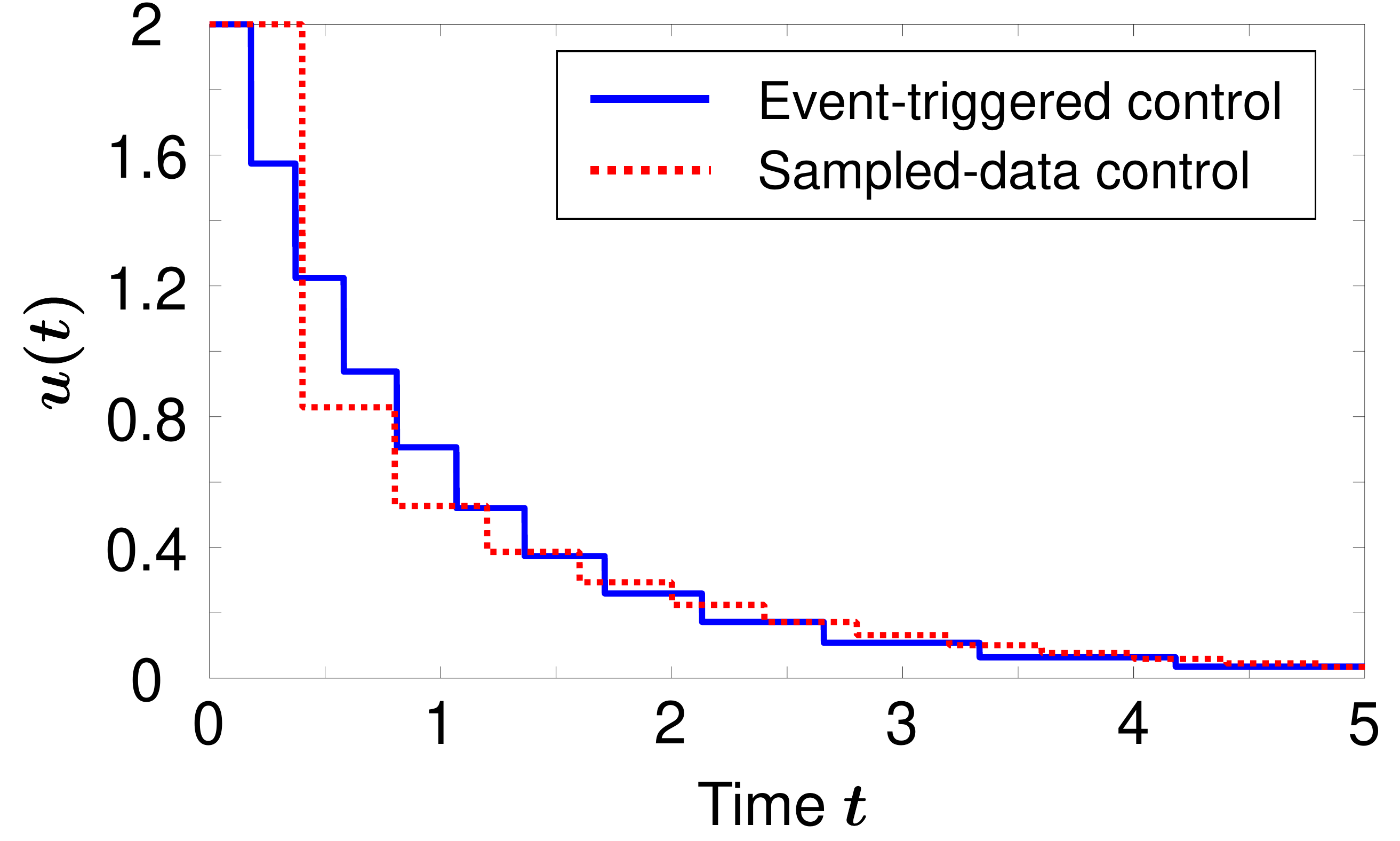}}
	\caption{Time response. \label{fig:response_ODE_PDE}}
\end{figure}


\subsection{Unbounded control}
As a numerical example in the case $B \not\in \mathcal{B}(U,X)$,
we apply the event-triggering mechanism in Theorem~\ref{thm:unbounded}
to an Euler-Bernoulli beam with structural damping \cite{Logemann2005}.
Let $\xi \in [0,1]$ and $t \geq 0$ denote the space and time variables.
We assume that the Euler-Bernoulli beam is hinged at the one end of the beam 
$\xi = 0$ and 
has a freely sliding clamped end at the other end $\xi = 1$.
Suppose that the shear force $u(t)$ is applied at $\xi = 1$.
The dynamics of the Euler-Bernoulli beam is given by
\begin{subequations}
	\label{eq:beam}
	\begin{align}
	\label{eq:DE}
	&\frac{\partial^2 z}{\partial t^2}(\xi,t) - 
	2 \gamma \frac{\partial^3 z}{\partial \xi^2 \partial t}(\xi,t) + 
	\frac{\partial^4 z}{\partial \xi^4}(\xi,t) =0, \quad \xi\in [0,1],~ t \geq 0\\
	\label{eq:BC}
	&z(0,t) = 0,\quad \frac{\partial^2 z}{\partial \xi^2}(0,t) = 0,\quad 
	\frac{\partial z}{\partial \xi}(1,t) = 0,\quad 
	-\frac{\partial^3 z}{\partial \xi^3}(1,t) = u(t),\quad t\geq 0,
	\end{align}
\end{subequations}
where $z(\xi,t)$ is the lateral deflection of the beam at location $\xi \in [0,1]$ and 
time $t\geq 0$, $u(t)$ is the input, and $\gamma \in (0,1)$
is the damping constant.

\subsubsection{Abstract evolution equation of Euler-Bernoulli beams}
We here recall 
the results developed in Section~5 of \cite{Logemann2005}:
an abstract evolution equation of the form \eqref{eq:state_equation}
for the PDE \eqref{eq:beam}.

We write $L^2(0,1)$ and $W^{4,2}(0,1)$ in place of 
$L^2([0,1], \mathbb{C})$ and $W^{4,2}([0,1], \mathbb{C})$, respectively.
We introduce the operator $A_0:D(A_0) \subset L^2(0,1) \to L^2(0,1)$,
\[
A_0 \zeta := \frac{d^4 \zeta}{d \xi^4}
\]
with domain
$
D(A_0) := \{
\zeta \in W^{4,2} (0,1):\zeta(0) = 0,~\zeta^{\prime \prime}(0) = 0,~
\zeta^{\prime}(1) = 0,~\zeta^{\prime \prime \prime }(1) = 0
\}.
$
We consider
the state space $X := D\big(A_0^{1/2}\big) \times L^2(0,1)$, which is 
a Hilbert space with the inner product
\[
\left\langle
\begin{bmatrix}
x_1 \\ x_2
\end{bmatrix},~
\begin{bmatrix}
y_1 \\ y_2
\end{bmatrix}
\right\rangle :=
\big\langle
A_0^{1/2}x_1,~
A_0^{1/2}y_1
\big\rangle_{L^2} + 
\left\langle
x_2,~
y_2
\right\rangle_{L^2}.
\]
The input space $U$ is given by $U := \mathbb{C}$.

Define  $e_{n} \in L^2(0,1)$ and $\lambda_n \in \mathbb{C}$ by
\begin{align*}
e_{n}(\xi) &:= \sqrt{2} \sin
\left(
-\frac{\pi}{2} + n\pi
\right) \xi \qquad \forall \xi \in [0,1],~\forall n \in \mathbb{N} \\
\lambda_{\pm n} &:= 
\left(-\gamma  \pm i\sqrt{1-\gamma^2} \right) \cdot \left(
-\frac{\pi}{2} + n\pi
\right)^2\qquad \forall n \in \mathbb{N}.
\end{align*}
The sets of functions $\{f_n \}_{n \in \mathbb Z^*}$
and $\{g_n \}_{n \in \mathbb Z^*}$
defined by
\begin{align*}
f_{\pm n} := 
\frac{\sqrt{2}}{1- \left(-\gamma \mp i \sqrt{1-\gamma^2} \right)^2}
\begin{bmatrix}
e_n/\lambda_{\pm n} \\
e_n
\end{bmatrix},
\quad 
g_{\pm n} &:= 
\frac{1}{\sqrt 2}
\begin{bmatrix}
-e_n/\lambda_{\mp n} \\
e_n
\end{bmatrix}
\qquad \forall n \in \mathbb{N}
\end{align*}
are biorthogonal, that is, 
\[
\langle
f_n,g_m
\rangle = 
\begin{cases}
1 & n = m \\
0 & n\not=m.
\end{cases}
\]
Using $\{f_n \}_{n \in \mathbb Z^*}$
and $\{g_n \}_{n \in \mathbb Z^*}$,
we define 
the operator $A:D(A) \subset X \to X$  by
\begin{align*}
A x := \sum_{n \in \mathbb{Z}^*} \lambda_{n} 
\langle
x,g_n
\rangle f_{n}\text{~~with~~} D(A):= 
\left\{
x \in X:
\sum_{n \in \mathbb{Z}^*} |\lambda_{n}|^2 \cdot
|\langle
x,g_n
\rangle|^2 < \infty
\right\},
\end{align*}
which is the generator of the
strongly continuous semigroup semigroup $T(t)$ given by
\[
T(t)x := \sum_{n \in \mathbb{Z}^*} e^{\lambda_{n} t}
\langle
x,g_n
\rangle f_{n}\qquad \forall x \in X,~\forall t \geq 0.
\]
Then $\sigma(A) = \{\lambda_{n}:n \in \mathbb{Z}^*\}$.
Let $X_{-1}$ denote 
the extrapolation space  associated with $T(t)$,
and define $B \in \mathcal{B}(U,X_{-1})$ by
\begin{equation*}
Bu := u \sum_{n \in \mathbb{Z}^*} (-1)^{|n| + 1} f_{n}\qquad 
\forall u \in \mathbb{C}.
\end{equation*}
Introducing the state vector
\[
x(t) := 
\begin{bmatrix}
z(\cdot, t) \\ \dfrac{\partial z}{\partial t}(\cdot, t)
\end{bmatrix},
\]
we can rewrite the PDE \eqref{eq:beam} 
in the form $\dot x(t) = Ax(t)+Bu(t)$, $t\geq0$.

\subsubsection{Numerical simulation}
Since $\sigma(A) = \{\lambda_{n}:n \in \mathbb{Z}^*\}$,
we see that Assumption~\ref{assump:finite_multiplicities}
is satisfied for every $\alpha < 0$.
We here
choose
$\alpha \in (-9\gamma \pi^2/4, -\gamma \pi^2/4)$ for the system decomposition
of Section~\ref{subsec:decomp} and check
whether or not Assumptions~\ref{assump:exponential_stability_T-} and 
\ref{assump:controllability} hold. For such $\alpha$, we obtain
$\sigma(A) \cap \overline{\mathbb{C}}_{\alpha} = \{\lambda_{-1},~\lambda_{1}\}$
and
$\omega(T^-) = -9\gamma \pi^2/4 < 0$.
The subspace $X^+:= \Pi X$ is spanned by 
$\{f_{-1},~f_1\}$, and using this basis, we can rewrite $A^+:= A|_{X^+}$ and $B^+ := \Pi B$ as
\[
A^+ = 
\begin{bmatrix}
\lambda_{-1} & 0 \\
0 & \lambda_1
\end{bmatrix},\quad 
B^+ = 
\begin{bmatrix}
1 \\ 1
\end{bmatrix}.
\]
Clearly, $(A^+, B^+)$ is controllable.
Thus Assumptions~\ref{assump:exponential_stability_T-} and 
\ref{assump:controllability} are satisfied.

Set $\gamma = 1/15$ and define the feedback operator 
$F$ by
\[
Fx := -\frac{13}{4} \gamma \pi^2 
\left(
\langle 
x,g_{-1}
\rangle 
+
\langle 
x,g_{1}
\rangle 
\right)\qquad \forall x \in X.
\]
Similarly to $A^+$ and $B^+$, we can rewrite $F^+:= F|_{X^+}$ with
respect to the basis $\{f_{-1}, f_1\}$ in the following way:
\[
F^+ = 
-\frac{13}{4} \gamma \pi^2
\begin{bmatrix}
1 & 1
\end{bmatrix}.
\]
We see from Proposition~\ref{prop:finite_event_trigger_cond} that 
if the threshold $\varepsilon > 0$ 
satisfies $\varepsilon \leq 0.83$, then
this finite-dimensional system 
\eqref{eq:finite_state_equation}
with the event-triggering mechanism \eqref{eq:event_trigger_present_finite}
is exponentially stable and its stability margin is at least
$7\gamma \pi^2/4$.

To compare event-triggered control and periodic sampled-data control,
we compute the time responses of 
the Euler-Bernoulli beam with the following initial state:
\[
z(\xi,0) = 1- \cos(\pi\xi),\quad
\frac{\partial z}{\partial t}(\xi,0) = 0\qquad \forall \xi \in [0,1].
\]
We apply the event-triggering 
mechanism \eqref{eq:event_trigger_present_finite} with
$\varepsilon = 0.70$ and $\tau_{\max} = 1$. For
periodic sampled-data control, we set $t_{k+1} - t_k \equiv 0.15$.
In the simulation, we approximate the state space $X$ by 
the linear span of $\{f_n: n \in \mathbb{Z}^*,~|n| \leq 15\}$.
Fig.~\ref{fig:z_norm} shows that the time responses of $\|x(t)\|$
are close between event-triggered control and periodic sample-data control.
However, we observe from Fig.~\ref{fig:input} that the difference between
the control input $u(t)$ under 
the event-triggering mechanism 
and that under the periodic mechanism is large, particularly, during the time interval $[0,1]$.
Since the relative change of $Fx(t)$ is small on $[0.3,0.8]$, the event-triggering mechanism
refrains from updating the input.

To check where or not the event-triggering mechanism can
reduce the number of the input updates,
we compute the  time $T_{\rm s} := \sup\{
t\geq 0: \|x(t)\| > 0.05 \|x^0\|
\}$
and count how many times the control input is updated on $(0,T_{\rm s})$.
Under the event-triggering 
mechanism \eqref{eq:event_trigger_present_finite},
we obtain $T_{\rm s} = 1.882$, and 
the control input is updated 
$10$ times on $(0,T_{\rm s})$.
On the other hand, under the periodic mechanism,
the  time $T_{\rm s}$ is $T_{\rm s} = 1.912$ and the number of control updates 
is $12$ on $(0,T_{\rm s})$.
Hence, the event-triggering mechanism achieves faster convergence with
less control updates than the periodic mechanism
in this example. 


\begin{figure}
	\centering
	\subcaptionbox{State norm $\|x(t)\|$.
		\label{fig:z_norm}}
	[.49\linewidth]
	{\includegraphics[width = 6.3cm,clip]{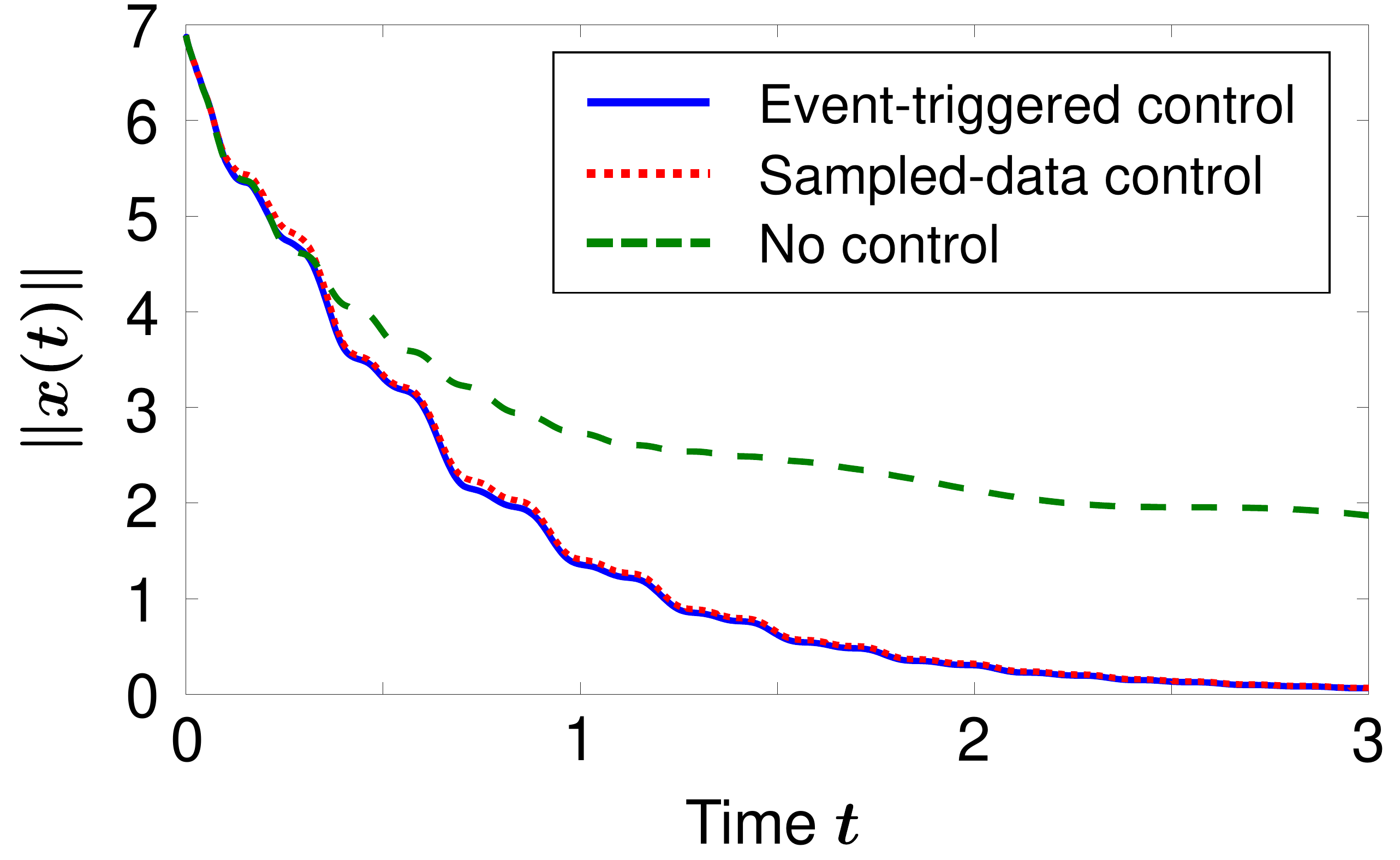}} 
	\subcaptionbox{Input $u(t)$.
		\label{fig:input}}
	[.5\linewidth]
	{\includegraphics[width = 6.3cm,clip]{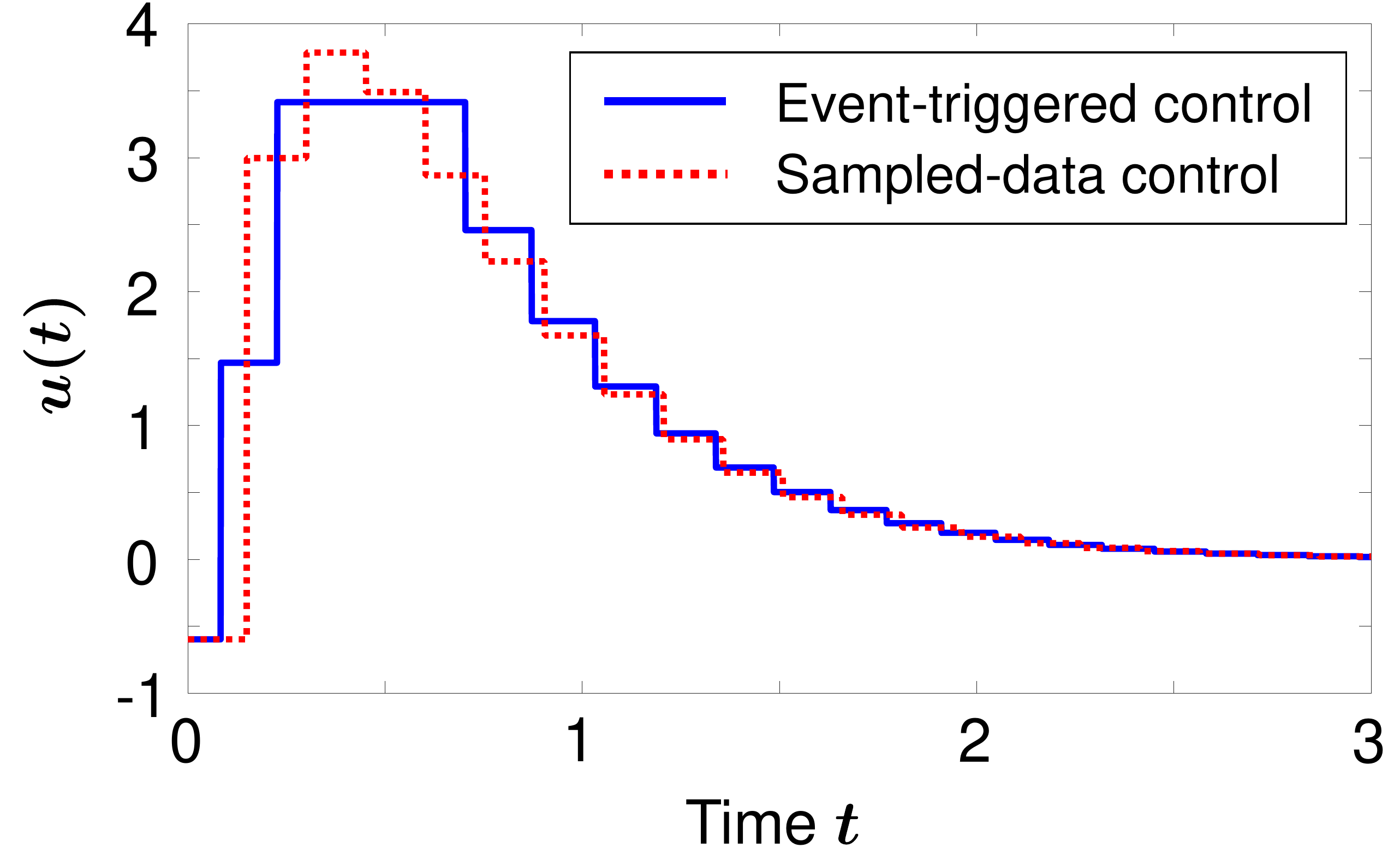}}
	\caption{Time response. \label{fig:response}}
\end{figure}

\section{Conclusion}
We have investigated the minimum inter-event time and the
exponential stability of infinite-dimensional event-triggered control systems.
We have employed the event-triggering mechanisms that 
compare the plant state and 
the error of the control input induced by the event-triggered implementation.
For these event-triggering mechanisms, 
we have shown that the minimum inter-event time is bounded from below 
by a strictly positive constant if 
the feedback operator is compact.
Moreover, we have obtained sufficient conditions on the threshold of
the event-triggering mechanisms for the exponential stability of the closed-loop system
with a bounded control operator.
For infinite-dimensional systems with unbounded control operators,
we have also analyzed the exponential stability of the closed-loop system under 
two event-triggering mechanisms, which 
are based on system decomposition and periodic event-triggering, respectively.
Future work involves extending the proposed analysis
to semi-linear infinite-dimensional systems.

\appendix
\section{Proof of Theorem~\ref{thm:continuity_initial_state}}
\label{sec:appendixA}
For $\tau \geq 0$,
define 
the operator $\Delta(\tau) \in \mathcal{B}(X)$  by 
$
\Delta(\tau) := T(\tau) + S_{\tau} F.
$
By the argument in Section 2.2,
if $F \in \mathcal{B}(X,U)$ is compact and if 
\eqref{eq:T1_coersive} holds for some $c_1>0$ and $s_1>0$, then 
there exists $\theta_{\rm m} >0$ such that
\[
\|F \Delta(\tau) x^0 - Fx^0\|_U \leq \varepsilon \| \Delta(\tau)x^0\|
\qquad \forall x^0 \in X,~\forall \tau\in [0,\theta_{\rm m}).
\]

The following lemma provides some properties of the operator $\Delta$, which
are useful to study the continuous dependence of solutions of the evolution equation 
\eqref{eq:plant}  under
the event-triggering mechanism \eqref{eq:time_seq_present} on initial states.
\begin{lemma}
	\label{lem:Delta_continuity}
	For any $\theta >0$, 
	the operator $\Delta(\tau) \in \mathcal{B}(X)$ satisfies 
	\[
	\sup_{0\leq \tau \leq \theta} \| \Delta(\tau)\|_{\mathcal{B}(X)}  < \infty.
	\]
	Moreover, if $F \in \mathcal{B}(X,U)$ is compact, then
	for every $x^0 \in X$ and every $\tau \geq 0$,
	\begin{align*}
	\lim_{\kappa \to 0}\Delta(\tau+\kappa)x^0 = \Delta(\tau) x^0,\qquad 
	\lim_{\kappa \to 0}\Delta(\tau+\kappa)^*x^0 = \Delta(\tau)^* x^0.
	\end{align*}
\end{lemma}
\begin{proof}
	This immediately follows from Lemma~\ref{lem:Stau} and 
	\begin{align*}
	\Delta(\tau+\tau_1)x - \Delta(\tau) x  &=
	T(\tau) \big(
	T(\tau_1) - I
	\big) x + \int_{\tau}^{\tau+\tau_1} T(s)BFx ds \\
	&= 
	T(\tau) (\Delta(\tau_1)x  -x)
	\end{align*}
	for every $x \in X$ and every $\tau,\tau_1\geq0$.
\end{proof}
The proof of Theorem~\ref{thm:continuity_initial_state} is based on the following lemma:
\begin{lemma}
	\label{lem:updating_time}
	Assume that $B \in \mathcal{B}(U,X_{-1})$ and that $F \in \mathcal{B}(X,U)$ is compact.
	Assume further that 
	the strongly continuous 
	semigroup $T(t)$ on $X$ satisfies \eqref{eq:T1_coersive} for some $c_1 > 0$ and $s_1 >0$.
	Let $x^0 \in X$ and $t_0 \geq 0$.
	Suppose that 
	\begin{equation}
	\label{eq:t1_def}
	t_1 := \inf \{ t > t_0: 
	\|F \Delta(t-t_0) x^0 - Fx^0\|_U > \varepsilon \| \Delta(t-t_0)x^0\|
	\}
	\end{equation}
	satisfies $t_1 < \infty$.
	For every $\kappa_1 \in (0,\theta_{\rm m})$ and $\delta_1 >0$,
	there exist $\kappa_0 \in (0,\theta_{\rm m})$ and $\delta_0 >0$ such that 
	for every $\eta_0 \geq 0 $ and $\zeta^0 \in X$ satisfying
	\begin{gather}
	\label{eq:ts0_xzeta0}
	|t_0 - \eta_0| < \kappa_0,\qquad
	\|x^0 - \zeta^0\| < \delta_0,
	\end{gather}
	we obtain 
	\begin{gather}
	|t_1 - \eta_1| < \kappa_1,\qquad 
	\|\Delta(t_1-t_0) x^0  - \Delta(\eta_1-\eta_0) \zeta^0\| < \delta_1, \label{eq:xt1_ys2}
	\end{gather}
	where
	\[
	\eta_1 := \inf \{ t > \eta_0: 
	\|F \Delta(t-\eta_0) \zeta^0 - F\zeta^0\|_U > \varepsilon \| \Delta(t-\eta_0)\zeta^0\|
	\}.
	\]
\end{lemma}
\begin{proof}	
	Let $x^0 \in X$ and $t_0\geq 0$ be given, and suppose that 
	$t_1$ defined by 
	\eqref{eq:t1_def} satisfies $t_1 < \infty$. Let us show the first inequality of \eqref{eq:xt1_ys2}.
	By the definition of $t_1$,
	for every $\kappa_1 \in (0,\theta_{\rm m})$,
	there exist $\kappa \in [0,\kappa_1)$ and $\varepsilon_0 > 0$ such that 
	\begin{equation}
	\label{eq:objective_eq}
	\|F \Delta(t_1+\kappa-t_0) x^0 - Fx^0\|_U =  \varepsilon \| \Delta(t_1+\kappa-t_0)x^0\| + \varepsilon_0.
	\end{equation}

	Let  
	$\eta_0 \geq 0$ and $\zeta^0 \in X$ satisfy
	\eqref{eq:ts0_xzeta0} for some $\kappa_0 \in (0,\theta_{\rm m})$ and $\delta_0 >0$. 
	Since 
	$t_1 \geq t_0 + \theta_{\rm m}$, it follows that
	$\eta_0 < t_0 +\kappa_0 < t_1$.
	We obtain
	\begin{align}
	&\|F \Delta(t_1+\kappa-t_0) x^0 - Fx^0\|_U \notag\\
	&\qquad \leq
	\| F \Delta(t_1+\kappa-\eta_0)\zeta^0 - F\zeta^0  \|_U +
	\|F (\Delta(t_1+\kappa-t_0) - I) (x^0 - \zeta^0)\|_U
	\notag\\
	&\quad \qquad + 
	\| F \Delta(t_1+\kappa-\eta_0)\zeta^0  - F\Delta(t_1+\kappa-t_0)\zeta^0  \|_U  
	\notag \\
	&\qquad <
	\| F \Delta(t_1+\kappa-\eta_0)\zeta^0 - F\zeta^0 \|_U + \phi_1(\kappa_0,\delta_0),
	\label{eq:used_eq1}
	\end{align}
	where \begin{align*}
	\phi_1(\kappa_0,\delta_0) &:=
	\delta_0\|F (\Delta(t_1+\kappa-t_0) - I)\|_{\mathcal{B}(X,U)} \\
	&\qquad+ 
	(\| x^0 \|+\delta_0) \| F \Delta(t_1+\kappa-\eta_0) - F\Delta(t_1+\kappa-t_0)  \|_{\mathcal{B}(X,U)}.
	\end{align*} 
	Lemmas~\ref{lem:uniform_conv} and \ref{lem:Delta_continuity} show that 
	\begin{align*}
	\lim_{\eta_0 \to t_0 }\| F \Delta(t_1+\kappa-\eta_0) - F\Delta(t_1+\kappa-t_0)  \|_{\mathcal{B}(X,U)} = 0.
	\end{align*}
	Therefore, $\phi_1(\kappa_0,\delta_0)$
	converges to zero
	as $(\kappa_0,\delta_0) \to (0,0)$.
	Similarly,
	\begin{align}
	\|\Delta(t_1+\kappa - \eta_0) \zeta^0\| 
	&<
	\|\Delta(t_1+\kappa - t_0) x^0\| + 
	\phi_2(\kappa_0,\delta_0),
	\label{eq:used_eq2} 
	\end{align}
	where
	\[
	\phi_2(\kappa_0,\delta_0) := 	\delta_0
	\|\Delta(t_1+\kappa - \eta_0)\|_{\mathcal{B}(X)} + 
	\|
	\Delta(t_1+\kappa - \eta_0)x^0-  \Delta(t_1+\kappa - t_0)  x^0
	\|.
	\]
	Using Lemma~\ref{lem:Delta_continuity} again, we find that 
	$\phi_2(\kappa_0,\delta_0)$
	converges to zero 
	as $(\kappa_0,\delta_0) \to (0,0)$.
	Combining \eqref{eq:objective_eq}--\eqref{eq:used_eq2}, we obtain
	\begin{align*}
	\| F \Delta(t_1+\kappa-\eta_0) \zeta^0 - F\zeta^0 \|_U &>
	\varepsilon \|\Delta(t_1+\kappa - \eta_0) \zeta^0\| \\
	&\qquad +
	\big(\varepsilon_0 - \phi_1(\kappa_0,\delta_0) - \varepsilon \phi_2(\kappa_0,\delta_0) \big).
	\end{align*}

	By the argument above,
	there exist $\kappa_0 \in (0,\theta_{\rm m})$ and $\delta_0 >0$ such that 
	for every	$\eta_0 \geq 0$ and $\zeta^0 \in X$ satisfying
	\eqref{eq:ts0_xzeta0},
	\[
	\| F \Delta(t_1+\kappa-\eta_0)\zeta^0 - F\zeta^0 \|_U >
	\varepsilon \|\Delta(t_1+\kappa - \eta_0) \zeta^0\|.
	\]
	Hence $\eta_1 < t_1 + \kappa_1$ by definition. In the same way, we obtain 
	$t_1 < \eta_1 + \kappa_1$.
	Thus the first inequality of \eqref{eq:xt1_ys2} holds.
	
	Next we prove the second inequality of \eqref{eq:xt1_ys2}.
	Let $\kappa_1 \in (0,\theta_{\rm m})$ and $\delta_1 >0$ be given. We have shown that
	there exist $\kappa_0 \in (0,\theta_{\rm m})$ and $\delta_0 >0$ such that 
	$|t_1 - \eta_1| < \kappa_1$
	for every $\eta_0 \geq 0 $ and $\zeta^0 \in X$ satisfying \eqref{eq:ts0_xzeta0}.
	Since $\eta_1 >t_1 - \kappa_1 > t_0$, it follows that
	\begin{align*}
	\|
	\Delta(t_1-t_0) x^0 - \Delta(\eta_1-\eta_0) \zeta^0
	\|
	&<
	\|
	\Delta(t_1 -t_0)x^0 - \Delta(\eta_1 - t_0)x^0
	\| \\
	&\qquad + 
	\|
	\Delta(\eta_1 -t_0)x^0 - \Delta(\eta_1 - \eta_0)x^0
	\|		\\
	&\qquad + 
	\delta_0\|\Delta(\eta_1 - \eta_0)\|_{\mathcal{B}(X)} .
	\end{align*}
	Lemma~\ref{lem:Delta_continuity}  shows that 
	\begin{align*}
	\lim_{\eta_1 \to t_1} 
	\|
	\Delta(t_1 -t_0)x^0 - \Delta(\eta_1 - t_0)x^0
	\| &= 0 \\
	\lim_{\eta_0 \to t_0}
	\|
	\Delta(\eta_1 -t_0)x^0 - \Delta(\eta_1 - \eta_0)x^0
	\|&=0.
	\end{align*}
	Moreover, 
	\[
	\|\Delta(\eta_1 - \eta_0)\|_{\mathcal{B}(X)} \leq
	\sup_{0\leq \tau \leq t_1+2\theta_{\rm m} - t_0} \|\Delta(\tau)\|_{\mathcal{B}(X)} < \infty.
	\]
	Thus we obtain the second inequality of \eqref{eq:xt1_ys2} for all sufficiently small
	$\kappa_0 \in (0,\theta_{\rm m})$ and $\delta_0 >0$.
\end{proof}

We are now in a position to show that 
solutions of the evolution equation \eqref{eq:plant}
continuously depend on 
initial states under the 
event-triggering mechanism \eqref{eq:time_seq_present}.
\begin{proof}[Proof of Theorem~\ref{thm:continuity_initial_state}]
	Let $x^0 \in X$ and $t_{\rm e} >0$ be given.
	Theorem~\ref{thm:no_zeno_present} shows that 
	there exist at most finitely many
	updating instants of the input $u$ induced by the
	event-triggering mechanism \eqref{eq:time_seq_present}
	on the interval $(0,t_{\rm e}]$.
	We first  assume that on the interval $(0,t_{\rm e}]$, there exist no 
	updating instants of the input $u$ derived from the initial state $x^0$.
	Then
	\[
	\inf
	\{ t > t_0: 
	\|F \Delta(t-t_0) x^0 - Fx^0\|_U > \varepsilon \| \Delta(t-t_0)x^0\|\} > t_{\rm e}.
	\]
	Note that $t_0 = \eta_0 = 0$.
	By  Lemma~\ref{lem:updating_time},
	there exists $\delta_0 >0$ such that for every $\zeta^0 \in X$ satisfying
	$\|x^0 - \zeta^0\| < \delta_0$, 
	the input $u$ derived from the initial state $\zeta^0$ has no 
	updating instants
	on the interval $(0,t_{\rm e}]$.
	Then 
	\[
	\|x(t) - \zeta(t)\| = \|\Delta(t)x^0 - \Delta(t)\zeta^0\| < \delta_0
	\sup_{0\leq t \leq t_{\rm e}}\|\Delta(t)\|_{\mathcal{B}(X)} \qquad \forall t \in [0,t_{\rm e}],
	\]
	and hence we obtain the desired conclusion by Lemma~\ref{lem:Delta_continuity}.
	
	Let $t_1,\dots,t_p \in (0,t_{\rm e}]$ with $t_1 < \dots < t_p$
	be the updating instants of the input $u$ derived from
	the initial state $x^0$.
	Without loss of generality, we may assume $t_p < t_{\rm e}$, since otherwise
	we shift $t_{\rm e}$ slightly so that  $t_p < t_{\rm e}$.
	Using Lemma~\ref{lem:updating_time} iteratively,
	we find that 
	for every $\kappa_1 \in (0,\theta_{\rm m})$ and $\delta_1 >0$,
	there exists $\delta_0 > 0$ such that 
	for every $\zeta^0 \in X$ satisfying $\|x^0 - \zeta^0\| < \delta_0$, the input $u$ 
	derived from the 
	initial state $\zeta^0$ 
	has $p$ updating instants $\eta_1,\dots,\eta_p \in (0,t_{\rm e})$ with $\eta_1<\dots<\eta_p$ and 
	\begin{equation}
	\label{eq:ell_bound}
	|t_{\ell} - \eta_{\ell}| < \kappa_1,\qquad \|x(t_{\ell}) - \zeta(\eta_{\ell})\| < \delta_1
	\qquad \forall \ell\in \{1,\dots,p\}.
	\end{equation}
	
	Choose $\ell\in\{1,\dots,p\}$ arbitrarily.
	Assume first that $\eta_{\ell} < t_{\ell}$.
	For every $\tau \in [0,t_{\ell} -\eta_{\ell})$, 
	\begin{align*}
	\|x(\eta_{\ell} + \tau) - \zeta(\eta_{\ell} + \tau) \| 
	< \| 
	x(t_{\ell}) - x(\eta_{\ell} + \tau)
	\| + 
	\|
	\zeta(\eta_{\ell}+\tau) - \zeta(\eta_{\ell})
	\|
	+\delta_1.
	\end{align*}
	We obtain
	\[
	\| 
	x(t_{\ell}) - x(\eta_{\ell} + \tau)
	\|=
	\|
	\Delta(t_\ell - t_{\ell-1})x(t_{\ell-1}) - \Delta(\eta_\ell + \tau - t_{\ell-1})x(t_{\ell-1})
	\|	
	\]
	and
	\begin{align*}
	\|
	\zeta(\eta_{\ell}+\tau) - \zeta(\eta_{\ell})
	\|	
	< \delta_1 \|\Delta(\tau) - I\|_{\mathcal{B}(X)} + 
	\|
	\Delta(\tau) x(t_\ell) - x(t_\ell)
	\|.
	\end{align*}
	Since $0 \leq \tau < t_{\ell} - \eta_{\ell}$, it follows that $|\tau| < \kappa_1$ and 
	\[
	|(t_{\ell} - t_{\ell-1})  - (\eta_\ell + \tau - t_{\ell-1})| = |t_{\ell}   -\eta_\ell - \tau| < \kappa_1. 
	\]
	Using Lemma~\ref{lem:Delta_continuity}, we find that 
	for every $\delta >0$, there exist $\kappa_1 \in (0,\theta_{\rm m})$ and $\delta_1 >0$ such that 
	the inequalities \eqref{eq:ell_bound} imply 
	\[
	\|x(\eta_{\ell} + \tau) - \zeta(\eta_{\ell} + \tau) \| < \delta \qquad \forall \tau \in [0,t_{\ell} - \eta_{\ell}).
	\]
	We obtain a similar result
	for the case $t_\ell \leq \eta_\ell$.

	Define 
	\[
	z_{\ell}^{\max} := \max\{t_\ell, \eta_\ell \},\quad
	z_{\ell + 1}^{\min} := \min\{t_{\ell+1}, \eta_{\ell + 1} \}\qquad  \forall \ell \in \{0,\dots,p\},
	\]
	where  $t_{p+1} := t_{\rm e}$ and $\eta_{p+1} := t_{\rm e}$.
	It suffices to show that for every $\delta >0$, there exist $\kappa_1 \in (0,\theta_{\rm m})$ and $\delta_1 >0$ such that 
	the inequalities \eqref{eq:ell_bound} imply 
	\[
	\|x(z_{\ell}^{\max} + \tau) - \zeta(z_{\ell}^{\max} + \tau) \| <\delta \qquad \forall \tau \in [0,z_{\ell + 1}^{\min} - z_{\ell}^{\max}],~
	\forall \ell \in \{0,\dots,p \}.
	\]
	This follows from
	\begin{align*}
	\|x(z_{\ell}^{\max} + \tau) - \zeta(z_{\ell}^{\max} + \tau) \| 
	& <
	\|\Delta(z_{\ell}^{\max} +\tau - t_{\ell})x(t_{\ell})
	- \Delta(z_{\ell}^{\max} +\tau - \eta_{\ell})x(t_{\ell})
	\| \\
	&\qquad \qquad 
	+ \delta_1
	\|\Delta(z_{\ell}^{\max} +\tau - \eta_{\ell}) \|_{\mathcal{B}(X)}.
	\end{align*}
	In fact,
	\begin{align*}
	\|\Delta(z_{\ell}^{\max} &+\tau - t_{\ell})x(t_{\ell})
	- \Delta(z_{\ell}^{\max} +\tau - \eta_{\ell})x(t_{\ell})\| \\
	&\qquad =
	\big\|T(\tau)\big(\Delta\big(|t_\ell - \eta_\ell|\big)x(t_{\ell}) -x(t_{\ell})\big)\big\| \\
	&\qquad \leq \sup_{0\leq \tau \leq t_{\rm e}} \|T(\tau)\|_{\mathcal{B}(X)}
	\cdot
	\big\|\Delta\big(|t_\ell - \eta_\ell|\big) x(t_{\ell}) - x(t_{\ell}) \big\|
	\end{align*}
	for every $\tau \in [0,z_{\ell + 1}^{\min} - z_{\ell}^{\max}]$ and every 
	$\ell \in \{0,\dots,p \}$.
	Moreover, Lemma~\ref{lem:Delta_continuity} shows that 
	\begin{align*}
	\lim_{\eta_{\ell}\to t_{\ell}}
	\big\|\Delta\big(|t_\ell - \eta_\ell|\big)x(t_{\ell}) - x(t_{\ell}) \big\|= 0\qquad\forall \ell \in \{0,\dots,p \}
	\end{align*}
	and that
	\[
	\|\Delta(z_{\ell}^{\max} +\tau - \eta_{\ell}) \|_{\mathcal{B}(X)} \leq 
	\sup_{0\leq t\leq t_{\rm e}}\|\Delta(t)\|_{\mathcal{B}(X)} < \infty.
	\]
	for every $\tau \in [0,z_{\ell + 1}^{\min} - z_{\ell}^{\max}]$ and every
	$\ell \in \{0,\dots,p \}$. This completes the proof.
\end{proof}

\section*{Acknowledgments}
The authors are grateful to the associate editor and anonymous
reviewers whose comments greatly improved the paper.

\end{document}